\documentclass[10pt,smallextended,nospthms,a4paper]{svjour3}
\usepackage[UKenglish]{babel}
\usepackage{times}
\usepackage{latexsym, amssymb, amsfonts, enumerate}

\usepackage{amsmath}
\usepackage{amsthm}
\usepackage{amsopn}
\usepackage{hyperref}
\numberwithin{equation}{section}
\newcommand{\R}{\mathbb{R}}
\newcommand{\N}{\mathbb{N}}
\newcommand{\F}{\mathcal{F}}
\newcommand{\E}{\mathbb{E}}
\renewcommand{\P}{\mathbb{P}}
\renewcommand{\vec}[1]{\boldsymbol{#1}}
\theoremstyle{plain}
\newtheorem{anytheorem}{Theorem}[section]
\newtheorem{theorem}[anytheorem]{Theorem}
\newtheorem{lemma}[anytheorem]{Lemma}
\newtheorem{corollary}[anytheorem]{Corollary}
\theoremstyle{definition}
\newtheorem{definition}[anytheorem]{Definition}
\newtheorem{remark}[anytheorem]{Remark}
\newtheorem{example}[anytheorem]{Example}
\newtheorem*{assumptionB}{Assumption B}
\newtheorem*{assumptionAC}{Assumption AC}
\begin{document}
\titlerunning{Nonlinear stochastic evolution equations of second order with damping}
\title{Nonlinear stochastic evolution equations of\newline second order with damping}
\author{
Etienne Emmrich \and David \v{S}i\v{s}ka
\thanks{This work has been partially supported by the Collaborative Research Center 910, which is funded by the German Science Foundation.}
}
\authorrunning{E.~Emmrich, D.~\v{S}i\v{s}ka}
\institute{Etienne Emmrich \at Technische Universit{\"a}t Berlin, Institut f\"ur Mathematik \\ Stra{\ss}e des 17.\ Juni 136, 10623 Berlin, Germany \\ \email{emmrich@math.tu-berlin.de} \and
David \v{S}i\v{s}ka \at University of Edinburgh,
School of Mathematics, James Clerk Maxwell Building,\\
King's Buildings, Peter Guthrie Tait Road, Edinburgh EH9 3FD, United Kingdom \\
\email{d.siska@ed.ac.uk}}
\date{28th September 2016}
\maketitle
\begin{abstract}
Convergence of a full discretization of a second order stochastic evolution equation with nonlinear damping is shown and thus existence of a solution is established. The discretization scheme combines an implicit time stepping scheme with an internal approximation. Uniqueness is proved as well.
\end{abstract}
\section{Introduction}
In this article, a second order evolution equation with additive and multiplicative ``noise'' is considered.
Such equations were first studied by Pardoux~\cite{pardoux:thesis}.
The corresponding initial value problem may be written as
\begin{equation}
\label{eq:0a}
\ddot{u} + A\dot{u} + Bu = f + C(u,\dot{u}) \dot{W} \text{ in } (0,T), \ \dot{u}(0) = v_0, \ u(0) = u_0,
\end{equation}
where $\dot{W}$ is the ``noise'' and $T>0$ is given.
A variety of phenomena in physical sciences and engineering can be modelled using equations of the form~\eqref{eq:0a}.
If $K$ is the integral operator with $(Kw)(t) := \int_0^t w(s)ds$ for some function $w$ then the above problem is (with $\dot{u} = v$)
formally equivalent to
\begin{equation}
\label{eq:0b}
\dot{v} + Av + B\left(u_0 + Kv\right) = f + C\left(u_0 + Kv,v\right)\dot{W} \text{ in } (0,T), \ v(0) = v_0 .
\end{equation}

To give a more precise meaning to the above problem, let $(H, (\cdot,\cdot), |\cdot|)$ be a real Hilbert space identified with its dual $H^*$ and let $(V_A,\|\cdot\|_{V_A})$ and $(V_B,\|\cdot\|_{V_B})$ be real, reflexive,  separable Banach spaces that are densely and continuously embedded in $H$.
The main result will require, in addition, that $V_A$ is densely and continuously embedded in $V_B$ and so
\begin{equation*}
V_A \hookrightarrow V_B \hookrightarrow H = H^* \hookrightarrow V_B^* \hookrightarrow V_A^*
\end{equation*}
with $\hookrightarrow$ denoting dense and continuous embeddings.
We will use $\langle \cdot, \cdot \rangle$ to denote the duality pairing
between elements of some Banach space and its dual.
Moreover, let $(\Omega, \F, (\F_t)_{t\in [0,T]},\P)$
be a stochastic basis and
let $W = (W(t))_{t\in [0,T]}$ be an 
infinite dimensional Wiener process
adapted to the filtration $(\F_t)_{t\in [0,T]}$ and such that for any
$t, h\geq0$ the increment $W(t+h)-W(t)$ is independent of $\F_t$.

The exact assumptions will be stated in Section~\ref{sec:2}.
For now it suffices to say that
$B:V_B\times \Omega \to V_B^*$ is a linear, bounded, symmetric and strongly positive operator.
The operator $A:V_A\times \Omega \to V_A^*$ and, for
$j\in \N$, the operators $C_j: V_B \times V_A \times \Omega \to H$ are nonlinear, jointly satisfying appropriate coercivity and monotonicity-like conditions.
Furthermore, we assume that $A$ is hemicontinuous and satisfies a growth condition.
We write 
$C = (C_j)_{j \in \N}$ and assume that $C$ maps $V_B \times V_A \times \Omega$ 
into $l^2(H)$.
We consider the stochastic evolution equation
\begin{equation}
\begin{split}
\label{eq:1}
& v(t) + \int_0^t \big[Av(s) + B\big(u_0 + (Kv)(s) \big)\big] ds\\
& = v_0 + \int_0^t f(s) ds + \int_0^t C\big( u_0 + (Kv)(s) , v(s)\big) dW(s)
\end{split}
\end{equation}
for $t \in [0,T]$, where $u_0$ and $v_0$ are given $\F_0$-measurable random variables that are $V_B$ and $H$-valued, respectively.
The $V_A^*$-valued process $f$ is adapted to $(\F_t)_{t\geq 0}$ and the stochastic integral is the It\^o integral with
\begin{equation*}
 \int_0^t C(u(s), v(s)) dW(s) =  \sum_{j=1}^{\infty} \int_0^t C_j (u(s),v(s)) dW_j(s).
\end{equation*}

Stochastic partial differential equations of second order in time are an active area of research.
Broadly speaking, difficulties arise from nonlinear operators, lack of damping, multiplicative noise and noise terms that are not continuous martingales as well as from regularity issues inherent to second order evolution equations.
Nonlinear operators are a particular issue if they are nonlinear in the ``highest order'' term rather than a nonlinear perturbation of a linear principal part.
We briefly point the reader to various papers exploring some of the above issues.

Peszat and Zabczyk~\cite{peszat:zabczyk:nonlinear:stochastic:wave} give necessary and sufficient conditions for the existence of solutions to a stochastic wave equation without damping, linear in the highest order term with nonlinear zero order term and nonlinear multiplicative noise.
Marinelli and Quer-Sardanyons~\cite{marinelli:existence} prove existence of solutions for a class of semilinear stochastic wave equations driven by an additive noise term given by a possibly discontinuous square integrable martingale.
Kim~\cite{kim:on:the:stochastic} proved existence and uniqueness of a solution to a semilinear stochastic wave equation with damping and additive noise.
Carmona and Nualart~\cite{carmona:nualart:random} investigate the smoothness properties of the solutions of one-dimensional wave equations with nonlinear random forcing.
Further work has been done regarding the smoothness of solutions, we refer the reader to Millet and Morien~\cite{millet:morien:on:a:stochastic:wave} as well as Millet and Sanz-Sol{\'e}~\cite{millet:sanz-sole:a:stochastic:wave} and the references therein.

In the deterministic case, second order evolution equations similar to (\ref{eq:0a}) have been investigated in the seminal paper of Lions and Strauss~\cite{Lions-Strauss}. This has been extended to the stochastic case by Pardoux~\cite{pardoux:thesis}. Indeed, Pardoux~\cite{pardoux:thesis} has shown existence of solutions via a Galerkin approximation and uniqueness to~\eqref{eq:1} under the assumption that the operators are deterministic and Lipschitz continuous on bounded subsets but allowing time-dependent operators.
Finally, we note that Pardoux~\cite{pardoux:thesis} also covers the case of
first-order-in-time stochastic evolution equations.
For first-order-in-time stochastic evolution equations, we also refer the reader to Krylov and Rozovskii~\cite{krylov:rozovskii:stochastic}.

Our aim is twofold: We wish to prove convergence of a fully discrete approximation of (\ref{eq:1}) including a time discretization. As far as the authors are aware, this paper is the first to prove convergence of a full discretization of stochastic evolution equations of second order
with a damping that has nonlinear principal part and a rather general multiplicative noise.
Moreover, we wish to extend Pardoux's result to random operators removing the Lipschitz-type condition. 
See Example~\ref{example:no_lip_on_bdd_subsets} for a situation 
where the assumption of Lipschitz continuity on bounded subsets does not hold
but the assumptions of this paper are satisfied.
We show existence of solutions to~\eqref{eq:1} by proving appropriate convergence of solutions to a full discretization. Unfortunately, the randomness of the operators finally requires the assumption that $V_A$ is continuously embedded in $V_B$ (see also Remark~\ref{rem:VA}), which is not the case with Pardoux~\cite{pardoux:thesis}. The reason is the use of the standard It\^o formula for the square of the norm,
see, e.g., Krylov and Rozovski{\u\i}~\cite{krylov:rozovskii:stochastic},
Gy\"ongy and Krylov~\cite{gyongy:krylov:ito:formula}
or Pr\'ev\^ot and R\"ockner~\cite{prevot:rockner:concise}.
It is left for future work whether the It\^{o} formula can be adapted to the general case where neither is $V_A$ embedded into $V_B$ nor is $V_B$ embedded into $V_A$. This is a rather delicate problem already for the integration by parts in the deterministic case (see again Lions and Strauss~\cite{Lions-Strauss} as well as Emmrich and Thalhammer \cite{emmrich:doubly}).
Finally, we will show that two solutions are indistinguishable.

Let us now describe the full discretization. A Galerkin scheme $(V_m)_{m\in \N}$ for $V_A$ will provide the internal approximation. For the temporal discretization, we choose an explicit scheme for approximating the stochastic integral but otherwise we use an implicit scheme.
Finally, we have to truncate the infinite
dimensional noise term.

Fix $m, r, N \in \N$.
Let $\tau := T/N$.
For $n=0,1,\ldots, N$, let $t_n := n\tau$.
Define $C^r := (C^r_j)_{j\in \N}$ with $C^r_j:= C_j$ for $j=1,\ldots,r$, $C^r_j = 0$ for $j>r$
and let
\begin{equation*}
\Delta W^n := \left\{
\begin{array}{ll}
W(t_n) - W(t_{n-1}) &\,\, \text{for } \,\,n=2,\ldots,N,\\
0,& \,\,  \text{for } \,\, n = 1.
\end{array}
\right.
\end{equation*}
For $g\in l^2(H)$, we define $g W(t) := \sum_{j\in \N} g_k W_k(t)$.
Clearly, $\tau$, $t_n$ and $\Delta W^n$ all depend on $N$.
This dependence will always be omitted in our notation.
The reason for taking $\Delta W^1 = 0$ will become clear during the proof of the a priori estimate for the discrete problem.
It allows one to assume that $v_0$ is an $H$-valued $\F_0$-measurable random variable (rather than a $V_A$-valued one).
This is consistent with the case of deterministic second-order-in-time evolution equations, see Lions and Strauss~\cite{Lions-Strauss}, and the stochastic second-order-in-time evolution equations, see Pardoux~\cite{pardoux:thesis}.

We now define $(u^n)_{n=0}^N$ and $(v^n)_{n=0}^N$ which will be approximations of $u$ and $v$, respectively, such that $u(t_n) \approx u^n$ and $v(t_n) \approx v^n$.
Assume that the $\F_0$-measurable random variables $u^0$ and $v^0$ take values in $V_m$ and are some given approximations of the initial values $u_0$ and $v_0$, respectively.
Let $(f^n)_{n=1}^N$ be an approximation of $f$ with $f^n$ being an $\F_{t_n}$-measurable $V_A^*$-valued random variable for $n=1,\ldots,N$.

Now we can fully discretize~\eqref{eq:1}.
We do this by approximating the integrands in~\eqref{eq:1} by piecewise constant processes on the time grid $(t_n)_{n=0}^N$.
Effectively, the value on the right-hand side of each interval is taken when approximating the non-stochastic integrals and the value on the left-hand side of each interval is taken when approximating the It\^o stochastic integral.
We define $(v^n)_{n=1}^N $ with $v^n$ being $V_m$-valued for $n=1,\ldots,N$ as the solution of
\begin{equation}
\begin{split}
\label{eq:2b}
& (v^n,\varphi)  + \tau\sum_{k=1}^n\bigg\langle A v^k +  B\bigg(u^0 + \tau \sum_{j=1}^k v^j\bigg),\varphi \bigg\rangle \\
& = (v^0, \varphi) + \tau \sum_{k=1}^n \langle f^k, \varphi\rangle + \sum_{k=1}^n \bigg( C^r\bigg(u^0 + \tau \sum_{j=1}^{k-1} v^j,v^{k-1}\bigg) \Delta W^k,\varphi \bigg)
\end{split}
\end{equation}
for all $\varphi \in V_m$ and $n= 1, \ldots , N$.
We can immediately see that~\eqref{eq:2b} corresponds to
\begin{equation}
\label{eq:2bb}
\begin{split}
& \bigg(\frac{v^n - v^{n-1}}{\tau}, \varphi\bigg) + \bigg\langle A v^n +  B\bigg(u^0 + \tau \sum_{k=1}^n v^k\bigg),\varphi \bigg\rangle\\
& = \langle f^n, \varphi \rangle + \bigg(C^r\bigg(u^0 + \tau \sum_{k=1}^{n-1} v^k,v^{n-1}\bigg)\frac{\Delta W^n}{\tau}, \varphi\bigg)
\end{split}
\end{equation}
for all $\varphi \in V_m$ and for $n=1,\ldots,N$.
This is exactly the numerical scheme one could obtain directly from~\eqref{eq:0b}.
In the case $C=0$ (i.e., the non-stochastic case) this would be an implicit Euler scheme in the ``velocity'', with the integral operator replaced by a simple quadrature.
With $u^n := u^0 + \tau\sum_{k=1}^n v^k$,
we further see that~\eqref{eq:2b} is also equivalent to
\begin{equation*}
\begin{split}
&\bigg(\frac{u^n - 2u^{n-1} + u^{n-2}}{\tau^2},\varphi\bigg) + \bigg\langle A\bigg(\frac{u^n-u^{n-1}}{\tau}\bigg) + B u^n,\varphi \bigg\rangle\\
& =  \langle f^n, \varphi \rangle 
+ \bigg(C^r\bigg(u^{n-1},\frac{u^{n-1}-u^{n-2}}{\tau}\bigg)\frac{\Delta W^n}{\tau}, \varphi \bigg)
\end{split}
\end{equation*}
for all $\varphi \in V_m$ and for $n=1,\ldots,N$, where $u^0$ and $u^{-1}:= u^0 - \tau v^0$ are given.
One could obtain this scheme directly from~\eqref{eq:0a}.

Numerical schemes for deterministic evolution equations of the above type have been investigated mostly for the particular case that $V_A = V_B$.  Emmrich and Thalhammer~\cite{emmrich:thalhammer:convergence} have proved weak convergence of time discretizations under the assumption that $V_A$ is continuously embedded in $V_B$. In Emmrich and Thalhammer \cite{emmrich:doubly},
weak convergence of fully discrete approximations is proved in the case when strongly continuous perturbations are added to the nonlinear principal part $A$ and the linear principal part $B$ even if $V_A$ is not embedded in $V_B$. This also generalizes the existence result of Lions and Strauss~\cite{Lions-Strauss}.
The convergence results have subsequently been extended in Emmrich and \v{S}i\v{s}ka~\cite{emmrich:siska:full}.
The situation for linear principal part $A$ but nonlinear, non-monotone $B$ requires a different analysis and is studied in Emmrich and \v{S}i\v{s}ka~\cite{emmrich:evolution}.

Numerical solutions of second-order-in-time stochastic partial differential equations have also been studied but for semilinear problems.
Kov\'acs, Saedpanach and Larsson~\cite{kovacs:saedpanah:larsson:finite}
considered a finite element approximation of the linear stochastic wave
equation with additive noise using semigroup theory.
Hausenblas~\cite{hausenblas:weak} demonstrated weak convergence (weak in the probabilistic sense) of numerical approximations to semilinear stochastic wave equations with additive noise.
De Naurois, Jentzen and Welti prove weak convergence rates for spatial spectral approximations
for an equation with multiplicative noise~\cite{deNaurois:jentzen:welti}.
For results on full-discretization, see also Anton, Cohen, Larsson and Wang~\cite{anton:cohen:larsson:wang}.
Semigroup theory is also used by Tessitore and Zabczyk~\cite{tessitore:zabczyk:wong} to prove weak convergence of the laws for Wong--Zakai approximations to semilinear strongly damped evolution equations of second order with multiplicative noise acting on the zero-order-in-time term.
Error estimates and estimates of the rate of convergence
can be found, e.g., in Walsh~\cite{walsh:on:numerical} and
Quer-Sardanyons and Sanz-Sol\'e~\cite{quer-sardanyons:sanz-sole:space} for particular examples governed by a linear principal part.

This paper is organized as follows.
Section~\ref{sec:2} contains all the assumptions and the statement of the main results of the paper.
In Section~\ref{sec:fulldisc}, we study the full discretization, prove that the fully discrete problem has a unique solution and establish a priori estimates.
We use the a priori estimates and compactness arguments in Section~\ref{sec:weaklimits} to obtain a stochastic process that is the weak limit of piecewise-constant-in-time prolongations of the solutions to the discrete problem.
In Section~\ref{sec:identlims}, it is shown that the weak limits satisfy the stochastic evolution equation. This finally proves convergence as well as existence of a solution. Uniqueness is then proved in Section~\ref{sec:uniq}. 
\section{Statement of assumptions and results}
\label{sec:2}
In this section, we state the precise assumptions on the operators, we define what is meant by a solution to~\eqref{eq:1} and we give the statement of the main result of this paper. Let us start with explaining the notation.

Throughout this paper, let $c > 0$ denote a generic constant that is independent of the discretization parameters. We set $\sum_{j=1}^0 z_j = 0$ for arbitrary $z_j$.
Recall that $T>0$ is given and that $(\Omega, \F, (\F_t)_{t\in [0,T]},\P)$
is a stochastic basis. By this, we mean that the probability space $(\Omega, \F, \P)$
is complete, $(\F_t)_{t\in [0,T]}$ is a filtration such that
any set of probability zero that is in $\F$ also
belongs to $\F_0$ and such that
$\F_s = \bigcap_{t>s} \F_t$  for all $s\in [0,T)$.
Moreover, $W = (W(t))_{t\in [0,T]}$ is an
infinite dimensional
Wiener process adapted to $(\F_t)_{t\in [0,T]}$ and such that for any $t , h\geq0$ the increment $W(t+h)-W(t)$ is independent of $\F_t$.

For a Banach space $(X,\|\cdot\|_X)$, we denote its dual by $(X^*,\|\cdot\|_{X^*})$ and we use $\langle g, w\rangle$ to denote the duality pairing between $g\in X^*$ and $w\in X$.
We will use the symbol $\rightharpoonup$ to denote weak convergence.
Let $p\in[2,\infty)$ be given and let $q  = \frac{p}{p-1}$ be the conjugate exponent of $p$.
For a separable and reflexive Banach space $X$, we denote
by $L^p(\Omega; X)$ and $L^p((0,T)\times \Omega; X)$ the standard Bochner--Lebesgue spaces (with respect to $\mathcal{F}$) and refer to Diestel and Uhl~\cite{diestel:vector} for more details. In particular, we recall that the concepts of strong measurability, weak measurability and measurability coincide since $X$ is separable (see also Amann and Escher~\cite{amann:escher}).
The norms are given by
\begin{equation*}
\|w\|_{L^p(\Omega;X)} := \left(\E\|w\|_X^p \right)^{1/p}
\text{ and }
\|w\|_{L^p((0,T)\times \Omega;X)} := \left(\E\int_0^T \|w(t)\|_X^p dt \right)^{1/p} .
\end{equation*}
The duals of $L^p(\Omega; X)$ and
$L^p((0,T)\times \Omega; X)$ are identified with $L^q(\Omega; X^*)$
and $L^q((0,T)\times \Omega; X^*)$, respectively.
Let $\mathcal{L}^p(X)$ be the linear subspace of $L^p((0,T)\times \Omega; X)$ consisting of equivalence classes of $X$-valued
stochastic processes
that are measurable with respect to the progressive 
$\sigma$-algebra. Note that $\mathcal{L}^p(X)$ is closed.

We say that an operator $D:X\times \Omega \to X^*$ is weakly
measurable with respect to some $\sigma$-algebra $\mathcal{G} \subseteq \mathcal{F}$ if the real-valued random variable $\langle Dw, z \rangle$ is $\mathcal{G}$-measurable for any $w$ and $z$ in $X$,
i.e., $Dw:\Omega \to X^*$ is weakly* $\mathcal{G}$-measurable for all $w \in X$.

Recall that $(H, (\cdot,\cdot), |\cdot|)$ is a real, separable Hilbert space, identified with its dual.
By
$h \in l^2(H)$, we mean that $h=(h_j)_{j\in \N}$ with $h_j \in H$ for $j\in \N$ and 
$\sum_{j\in \N} |h_j|^2 < \infty$.
We define the inner product in $l^2(H)$ by $(g,h)_{l^2(H)} := \sum_{j\in\N} (g_j,h_j)$,
where $g , h \in l^2(H)$.
This induces a norm on $l^2(H)$ by $|h|_{l^2(H)} = (h,h)_{l^2(H)}^{1/2}$.
Further recall that $(V_A,\|\cdot\|_{V_A})$ and $(V_B,\|\cdot\|_{V_B})$ are real, reflexive and separable Banach spaces that are densely and continuously embedded in $H$
and that the main result will require, in addition, that $V_A$ is densely and continuously embedded in $V_B$ and so
\begin{equation}\label{embedding}
V_A \hookrightarrow V_B \hookrightarrow H = H^* \hookrightarrow V_B^* \hookrightarrow V_A^*
\end{equation}
with $\hookrightarrow$ denoting dense and continuous embeddings.
Our notation does not distinguish whether the duality pairing 
$\langle \cdot, \cdot \rangle$ is the duality pairing between $V_A$ and $V_A^*$ or
$V_B$ and $V_B^*$ since in situations when both would be well defined they coincide due to~\eqref{embedding}.

Finally, we need a Galerkin scheme for $V_A$ which we denote by $(V_m)_{m\in \N}$. That is, we assume that for all $m\in \N$ we have $V_m \subseteq V_{m+1} \subset V_A$ and that $\bigcup_{m\in \N} V_m$ is dense in $V_A$.
We assume further, without loss of generality, that the dimension of $V_m$ is $m$.

\begin{assumptionB}
Let $B:V_B\times \Omega \to V_B^*$ be
weakly $\mathcal{F}_0$-measurable.
Assume moreover that $B$ is, almost surely, {\em linear}, {\em symmetric} and let there be $\mu_B > 0$ and $c_B > 0$ such that, almost surely,
\begin{equation*}
\langle Bw ,w \rangle \geq \mu_B\|w\|^2_{V_B} \text{ and } \|Bw\|_{V_B^*} \leq c_B\|w\|_{V_B} \quad \forall w\in V_B.
\end{equation*}
This means that $B$ is, almost surely, {\em strongly positive} and {\em bounded}.
\end{assumptionB}
Note that with this assumption we can define, for $\P$-almost all $\omega \in \Omega$, an inner product on $V_B$ by $(w,z)_B := \langle Bw,z \rangle$ for any $w,z \in V_B$.
We will denote the norm associated with the inner product by $|\cdot|_B := (\cdot, \cdot)_B^{1/2}$. This norm is equivalent to $\|\cdot\|_{V_B}$.

\begin{assumptionAC}
The operators $A:V_A\times\Omega \to V_A^*$ and 
$C:V_B\times V_A\times \Omega \to l^2(H)$
are weakly $\F_0$-measurable.
Moreover, we assume that $A$, is almost surely, {\em hemicontinuous},
i.e., there is $\Omega_0 \in \F_0$ with $\P(\Omega_0) = 0$ and for every $\omega \in \Omega\setminus \Omega_0$ the function $\epsilon\mapsto \langle A(w+\epsilon z,\omega), v \rangle : [0,1] \to \mathbb{R}$ is continuous
for any $v,w,z \in V_A$.

There is $c_A >0 $ such that, almost surely, the {\em growth condition}
\begin{equation*}
  \|Aw\|_{V_A^*} \leq c_A(1+\|w\|_{V_A})^{p-1} \quad \forall w \in V_A
\end{equation*}
is satisfied.

There are $\mu_A > 0$, $\lambda_A \geq 0$, $\lambda_B \geq 0$ and $\kappa\geq 0$ such that, almost surely, the operators $A$ and $C$ satisfy the {\em monotonicity-like} condition
\begin{equation}
\label{eq:orig_monotonicity}
\langle Aw - Az, w-z \rangle + \lambda_A |w-z|^2 \geq \frac{1}{2}|C(u,w) - C(v,z)|_{l^2(H)}^2 - \lambda_B|u-v|_B^2
\end{equation}
for any $w,z \in V_A$ and $u,v \in V_B$
and the {\em coercivity-like} condition
\begin{equation}
\label{eq:orig_coercivity}
\langle Aw,w \rangle + \lambda_A|w|^2\geq \mu_A \|w\|_{V_A}^p + \frac{1}{2}|C(u,w)|_{l^2(H)}^2 - \lambda_B|u|_B^2 - \kappa
\end{equation}
for any $w \in V_A$ and $u \in V_B$.
\end{assumptionAC}

The almost sure hemicontinuity of $A:V_A \times \Omega \to V_A^*$ together with the almost sure monotonicity of $A + \lambda_A I:V_A \times \Omega \to V_A^*$ (see~\eqref{eq:orig_monotonicity}) imply
that $A$ is in fact, almost surely, demicontinuous (see also Krylov and Rozovskii~\cite{krylov:rozovskii:stochastic}).

The growth condition and coercivity from Assumption~{\em AC} imply that for any $u\in V_B$ and $w\in V_A$,
\begin{equation}
\label{eq:Cbdd}
|C(u,w)|_{l^2(H)}^2 \leq c(1+|u|_B^2 + |w|^2 + \|w\|_{V_A}^p).
\end{equation}
The monotonicity-like condition implies that $C$ is Lipschitz continuous in its first argument uniformly with respect to its second argument.
Indeed for all $w\in V_A$ and all $u,v \in V_B$ we get
\begin{equation*}
|C(u,w) - C(v,w)|_{l^2(H)} \leq \sqrt{2\lambda_B}|u-v|_B.
\end{equation*}
If the coercivity and monotonicity-like conditions are satisfied then
we obtain with
$\lambda := 2\max(\lambda_A, \lambda_B, \kappa)$
\begin{equation}
\label{eq:mod_monotonicity}
2\langle Aw - Az, w-z \rangle + \lambda|w-z|^2 + \lambda|u-v|_B^2 \geq |C(u,w) - C(v,z)|_{l^2(H)}^2
\end{equation}
and
\begin{equation}
\label{eq:mod_coercivity}
2\langle Aw,w \rangle + \lambda(|w|^2 + |u|_B^2 + 1) \geq 2\mu_A \|w\|_{V_A}^p + |C(u,w)|_{l^2(H)}^2.
\end{equation}

In many applications, the operators $A$ and $C$ would arise separately from various modelling considerations.
In such a situation, it may be useful to see under what assumptions on $A$ and $C$, stated independently, would (\ref{eq:orig_monotonicity}) and (\ref{eq:orig_coercivity}) hold.
To that end,  assume that there are $\mu_A > 0$ and $\lambda_1, \lambda_2 \ge 0$ such that, almost surely, for all $w,z \in V_A$
\begin{equation}
\label{eq:mon_and_coer_of_A_only}
\langle Aw - Az , w -z \rangle + \lambda_1 |w-z|^2 \geq 0 \,\,\textrm{ and }\,\, \langle Aw, w \rangle + \lambda_2 |w|^2 \geq \mu_A \|w\|_{V_A}^p.
\end{equation}
Assume further that there are $\lambda_3, \lambda_4 \geq 0$  such that, almost surely, for all $u,v \in V_B$ and $w,z \in V_A$
\begin{equation*}
|C(u,w) - C(v,z)|_{l^2(H)}^2 \leq  \lambda_3|u-v|_B^2 + \lambda_4|w-z|^2.
\end{equation*}
With $v=z=0$ and $\kappa = |C(0,0)|_{l^2(H)}^2$, we obtain
\begin{equation*}
|C(u,w)|_{l^2(H)}^2 \leq 2\big( \lambda_3 |u|_B^2 + \lambda_4 |w|^2 + \kappa \big).
\end{equation*}
Then \eqref{eq:orig_monotonicity} and \eqref{eq:orig_coercivity} follow with a suitable choice of the constants.

Examples of operators satisfying the above assumptions and the corresponding stochastic partial differential equations can be found in Pardoux~\cite[Part III, Ch. 3]{pardoux:thesis}.
Let us present an example where the condition on Lipschitz continuity on bounded sets 
as required by Pardoux is not satisfied but the assumptions of this paper hold.

\begin{example}
\label{example:no_lip_on_bdd_subsets}
We consider a bounded domain $\mathcal{D}$ in $\R^d$
with smooth boundary and take 
$V_A=V_B=H^1_0(\mathcal{D})$, the standard Sobolev space, and $H=L^2(\mathcal{D})$.
Following Emmrich~\cite{emmrich:time}, we consider 
$\rho:\R^d\to\R^d$ given by
\begin{equation*}
\rho(z) = \left\{ 
\begin{array}{ll}
0 & \text{if} \quad |z| = 0,\\
|z|^{-1/2}z & \text{if} \quad |z| \in (0,1),\\
z & \textrm{otherwise}.
\end{array}
\right.	
\end{equation*}
It is then easy to check that $A:V_A\to V_A^*$ 
given by
\begin{equation*}
\langle A v, w \rangle = \int_\mathcal{D} \rho(\nabla v)\cdot \nabla w \, dx	
\end{equation*}
satisfies the hemicontinuity and growth condition of Assumption AC 
as well as the monotonicity and coercivity 
condition~\eqref{eq:mon_and_coer_of_A_only}. 
Moreover it is possible to show that this operator $A$ does
not satisfy the assumption of Lipschitz continuity
on bounded subsets of Pardoux~\cite{pardoux:thesis}.
\end{example}

We say that $\tilde{z}$ is a modification of $z\in \mathcal{L}^\gamma(X)$
($\gamma \in [1,\infty)$) if $z(t,\omega) = \tilde{z}(t,\omega)$ for
$(dt \times d\P)$-almost all $(t,\omega)$.
If $X\hookrightarrow H$ then we say that $\tilde{z}$ is an $H$-valued
continuous modification of $z\in \mathcal{L}^\gamma(X)$
if $t \mapsto \tilde{z}(t,\omega) : [0,T] \to H$  is continuous for almost all $\omega \in \Omega$ and
$\tilde{z}$ is a modification of $z$.

We will use the following notation for stochastic integrals:
Given $x\in \mathcal{L}^2(H)$ and
$y\in \mathcal{L}^2(l^2(H))$, we write
\[
\int_0^t (x(s), y(s)dW(s)) := 
\sum_{j\in\N}\int_0^t (x(s), y_j(s))dW_j(s).
\]

\begin{definition}[Solution]
\label{def:soln}
Let $u_0 \in L^2(\Omega; V_B)$ and $v_0 \in L^2(\Omega; H)$ be $\F_0$-measurable and let $f\in \mathcal{L}^q\left({V_A}^*\right)$.
Let there be $v\in \mathcal{L}^p(V_A)$ such that $u_0 + Kv \in \mathcal{L}^2(V_B)$
and moreover let there be an $H$-valued continuous modification $\tilde{v}$ of $v$.
Then $v$ is said to be a {\em solution} to~\eqref{eq:1} if
$\P$-almost everywhere,
for all $t\in [0,T]$ and for all $z \in V_A$
\begin{equation*}
\begin{split}
&  ( \tilde{v}(t), z )  + \int_0^t \left\langle Av(s) + B\left(u_0 + (Kv)(s)\right),z\right\rangle ds \\
& = (v_0,z) + \int_0^t \langle f(s),z\rangle ds + \int_0^t \big(z, C(u_0 + (Kv)(s),v(s)) dW(s) \big).
\end{split}
\end{equation*}
\end{definition}

We will typically not distinguish between $\tilde{v}$ and $v$, denoting both by $v$, to simplify notation.
The following result on the uniqueness of solutions to~\eqref{eq:1}
will be proved in Section~\ref{sec:uniq}.

\begin{theorem}[Uniqueness of solution]
\label{lemma:uniq}
Let Assumptions~{\em AC} and {\em B} and let (\ref{embedding}) hold.
Let $v_1$ and $v_2$ be two solutions to~\eqref{eq:1}
in the sense of Definition~\ref{def:soln}.
Then
\begin{equation*}
\P\left(\max_{t \in [0,T]}|v_1(t) - v_2(t)| = 0\right) = 1,	
\end{equation*}
i.e., $v_1$ and $v_2$ are indistinguishable.
Moreover, if we let
\begin{equation*}
u_1 = u_0 + Kv_1 \quad \textrm{ and } \quad u_2 = u_0 +Kv_2
\end{equation*}
then
\begin{equation*}
\P\left(\max_{t \in [0,T]}\|u_1(t) - u_2(t)\|_{V_B} = 0\right) = 1,	
\end{equation*}
i.e., $u_1$ and $u_2$ are also indistinguishable.
\end{theorem}

Consider a sequence $(m_\ell,r_\ell,N_\ell)_{\ell \in \N}$
such that $m_\ell\to \infty$, $r_\ell \to \infty$ and $N_\ell \to \infty$ as $\ell \to \infty$ and let $\tau_\ell = T/N_\ell$.
Let $(u^0_\ell)_{\ell \in \N}$ be a sequence of $\F_0$-measurable random variables with values in $V_{m_\ell}$ such that $u^0_\ell \in L^2(\Omega;V_B)$ and $u^0_\ell \to u_0$ in $L^2(\Omega;V_B)$ as $\ell \to \infty$. Moreover, let $(v^0_\ell)_{\ell \in \N}$ be a sequence
of $\F_0$-measurable random variables with values in $V_{m_\ell}$ such that $v^0_\ell \in L^2(\Omega;H)$ and $v^0_\ell \to v_0$ in $L^2(\Omega;H)$ as $\ell \to \infty$.
For $f\in \mathcal{L}^q\left({V_A}^*\right)$, we use the approximation
\begin{equation}\label{approx-fff}
f^n := \frac{1}{\tau_\ell} \int_{t_{n-1}}^{t_n} f(t)\,dt,\quad n = 1, \ldots, N_\ell\,,
\end{equation}
where we recall that $t_n = n\tau_\ell$
for $n=0,\ldots,N_\ell$.
Note that for readability we drop the dependence of $t_n$ and $f^n$ on $N_\ell$.

For each $(m_\ell, r_\ell, N_\ell)$, we take $(f^n)_{n=1}^{N_\ell}$ and the solution to the scheme~\eqref{eq:2b} and use this to define stochastic processes $f_\ell$, $v_\ell$ and $u_\ell$, which will be approximations of $f$, $v$ and $u$, as follows: for $n=1,\ldots,N_\ell$, let
\begin{equation}
\label{eq:ext1}
f_\ell(t) := f^n, \, v_\ell(t) := v^n, \, u_\ell(t) := u^n \, \textrm { if } \, t\in (t_{n-1}, t_n].
\end{equation}
We may set $f_\ell(0) = f^1$, $v_\ell(0) = v^1$, $u_\ell(0) = u^1$.
Note that $u^n$ and $v^n$ indeed depend on $m_\ell$ and $N_\ell$.

We see that even if $v^n$ and $u^n$ are $\F_{t_n}$-measurable for each $n=0,1,\ldots, N_\ell$ then the processes $v_\ell$ and $u_\ell$ are not $(\F_t)_{t\in [0,T]}$ adapted.
Thus we will not be able to directly use compactness-based arguments to get weak limits that are adapted.
To overcome this, we will also use the following approximations:
for $n=2,\ldots,N_\ell$, let
\begin{equation}
\label{eq:ext2}
v_\ell^-(t) := v^{n-1}, \, u_\ell^-(t) := u^{n-1}\, \textrm { if } \, t\in [t_{n-1}, t_n)
\end{equation}
and let $v_\ell^{-}(t) = 0$ and $u_\ell^{-}(t) = u^0$ if $t\in [0,\tau_\ell)$.
We may set $v_\ell^-(T) = v^{N_\ell}$, $u_\ell^-(T) = u^{N_\ell}$.

We note that $v_\ell(t_n) = v_\ell^-(t_n) = v^n$ and $u_\ell(t_n) = u_\ell^-(t_n) = u^n$ for $n=1,\ldots,N_\ell$.
If $v^n$ and $u^n$ are $\F_{t_n}$-measurable for each $n=0,1,\ldots, N$ then the processes $v_\ell^-$ and $u_\ell^-$ are $(\F_t)_{t\in [0,T]}$ adapted.
For $v_\ell^-$ (and $u_\ell^-$) we will then be able to obtain weak limits that are themselves adapted processes.
Later, we will show that the weak limits of $v_\ell^-$ and $v_\ell$ as well as of $u_\ell^-$ and $u_\ell$ coincide.

We now rewrite~\eqref{eq:2b} in an integral form.
To that end, define $\theta_\ell^+(0) :=0$ and $\theta_\ell^+(t) := t_n$ if $t\in (t_{n-1},t_n]$ and $n=1,\ldots,N_\ell$.
Then saying $(v^n)_{n=1}^N$ satisfies~\eqref{eq:2b} with $m=m_\ell$ and $\tau = \tau_\ell$ is equivalent to
\begin{equation}
\label{eq:scheme_continuous_formulation}
\begin{split}
& (v_\ell(t),\varphi) + \bigg\langle\int_0^{\theta_\ell^+(t)} (Av_\ell(s)+Bu_\ell(s)-f_\ell(s))ds,\varphi \bigg\rangle\\
& = (v^0_\ell,\varphi)
+ \bigg(\int_{\tau_\ell}^{\theta_\ell^+(t)} C^{r_\ell}(u_\ell^-(s),v_\ell^-(s))dW(s),\varphi\bigg)
\end{split}
\end{equation}
for all $\varphi \in V_{m_\ell}$ and for all $t\in (0,T]$.

The following theorem is the main result of the paper.
Recall that $\lambda$ arises from Assumptions AC as $\lambda = 2\max(\lambda_A, \lambda_B, \kappa)$.
\begin{theorem}[Existence and convergence]
\label{thm:main}
Let Assumptions~{\em AC} and {\em B}  and let (\ref{embedding}) hold.
Let $u_0 \in L^2(\Omega; V_B)$ and $v_0 \in L^2(\Omega; H)$ be $\F_0$-measurable and let $f\in \mathcal{L}^q\left({V_A}^*\right)$.
Then the stochastic evolution equation~\eqref{eq:1} possesses a solution $v\in \mathcal{L}^p(V_A)$ according to Definition~\ref{def:soln} with $u=u_0+Kv \in \mathcal{L}^2(V_B)$.

Furthermore, consider $(m_\ell,N_\ell)_{\ell \in \N}$ with $m_\ell\to \infty$ and $N_\ell \to \infty$ as $\ell \to \infty$ such that $\sup_{\ell\in\mathbb{N}} \lambda \tau_\ell < 1$. Let $(u^0_\ell)_{\ell \in \N} \subset L^2(\Omega;V_B)$,  $(v^0_\ell)_{\ell \in \N} \subset L^2(\Omega;H)$ be sequences of
$\F_0$-measurable random variables with values in $V_{m_\ell}$ such that $u^0_\ell \to u_0$ in $L^2(\Omega;V_B)$ and $v^0_\ell \to v_0$ in $L^2(\Omega;H)$ as $\ell \to \infty$. Let $(f_\ell)_\ell\in \N$ be given by (\ref{approx-fff}) and (\ref{eq:ext1}). The numerical scheme \eqref{eq:scheme_continuous_formulation} then admits a unique solution with
\begin{equation*}
\begin{split}
& u_\ell \rightharpoonup u \text{ in } L^2((0,T)\times\Omega;V_B)
\text{ and } v_\ell \rightharpoonup v \text{ in } L^p((0,T)\times \Omega;V_A)\\
& 
u_\ell(T) \to u(T) \text{ in } L^2(\Omega;V_B) \text{ and } v_\ell(T) \to v(T) \text{ in } L^2(\Omega;H) \text{ as } \ell \to \infty .
\end{split}
\end{equation*}
\end{theorem}

The proof can be briefly summarized as follows:
We first need to show that the fully discretized problem has a unique solution, which is covered by Theorem~\ref{thm:existence_disc}.
Then we obtain a priori estimates for the fully discrete problem (Theorem~\ref{thm:disc-apriori}), so that we can extract weakly convergent subsequences using compactness arguments (Lemma~\ref{lemma:weaklimits}).
At this point, the only step left to do is to identify the weak limits from the nonlinear terms. Convergence of the full sequence of approximations (and not just of a subsequence) follows because of the uniqueness result.

\begin{remark}
\label{rem:VA}
Our results require the assumption that $V_A \hookrightarrow V_B$.
The need for this assumption arises from the use of the standard It\^o formula for the square
of the norm, which also provides existence of a continuous modification. 	
However, if $A$, $B$ and $C$ are deterministic then
Pardoux~\cite[Part III, Chapter 2, Theorem 3.1]{pardoux:thesis}
proves the energy equality~\eqref{eq:7} and sufficient regularity
without the need to assume $V_A \hookrightarrow V_B$.
It remains open whether this approach can be extended to the situation of random and time-dependent operators.
\end{remark}

\section{Full discretization: existence, uniqueness and a priori estimates}
\label{sec:fulldisc}
In this section, we show that the full discretization~\eqref{eq:2b} has a unique solution, adapted to the filtration given, and prove an a priori estimate.
The a priori estimate is essential for the proof of the main result of the paper as this allows us to use compactness arguments to extract weakly convergent subsequences from the sequence of approximate solutions.

Existence of solutions to the discrete problem will be proved by applying the following lemma.
\begin{lemma}
\label{lemma:consequenceofbrower}
Let $\vec{h}:\R^m \to \R^m$ be continuous. If there is $R>0$ such that $\vec{h}(\vec{v})\cdot \vec{v} \geq 0$ whenever $\|\vec{v}\|_{\R^m} = R$ then there exists $\vec{\bar{v}}$ satisfying $\|\vec{\bar{v}}\|_{\R^m} \leq R$ and $\vec{h}(\vec{\bar{v}}) = 0$.
\end{lemma}
\begin{proof}
The lemma is proved by contradiction from Brouwer's fixed point theorem (see, e.g.,~\cite[Ch. 3, Lemma 2.1]{ggz}).
\end{proof}
To obtain the appropriate measurability of the solution to the discrete problem we need the following lemma, which is a modification of Gy\"ongy~\cite[Lemma 3.8]{gyongy:on:stochastic:III}.
\begin{lemma}
\label{lemma:measurability}
Let $(S,\Sigma)$ be a measure space.
Let $\vec{f}:S\times \R^m \to \R^m$ be a function
that is $\Sigma$-measurable in its first argument for every $\vec{x}\in \R^m$,
that is continuous in its second argument for every $\alpha \in S$
and moreover such that
for every $\alpha \in S$ the equation $\vec{f}(\alpha, \vec{x}) = \vec{0}$ has a unique solution $\vec{x}=\vec{g}(\alpha)$. Then $\vec{g}:S \to \R^m$ is $\Sigma$-measurable.
\end{lemma}
\begin{proof}
Let $F$ be a closed set in $\R^m$.
Then
\begin{equation*}
\vec{g}^{-1}(F) := \{\alpha \in S : \vec{g}(\alpha) \in F\} = \left\{\alpha \in S : \min_{\vec{x}\in F}\|\vec{f}(\alpha, \vec{x})\|_{\mathbb{R}^m} = 0 \right\},
\end{equation*}
since $F$ is closed.
But since $\vec{f} = \vec{f}(\alpha,\vec{x})$ is continuous in the second argument for every $\alpha \in S$ and $\Sigma$-measurable in the first argument for every $\vec{x} \in \mathbb{R}^m$, we see that $\vec{g}^{-1}(F) \in \Sigma$.
\end{proof}

Let $W^r := (W^r_j)_{j\in \N}$ and $\Delta W^{r,n}:= (\Delta W^{r,n}_j)_{j\in \N}$
with 
\[
W^r_j := \left\{
\begin{array}{lll}
	W_j & \text{for} & j=1,\ldots,r\,, \\
	0   &  \text{for} & j > r
\end{array}
\right.
\quad \text{and}\quad 
\Delta W^{r,n}_j := \left\{
\begin{array}{lll}
	\Delta W^n_j & \text{for} & j=1,\ldots,r\,, \\
	0   &  \text{for} & j > r.
\end{array}
\right.
\]

We are now ready to prove existence of solutions to the full discretization.
\begin{theorem}[Existence and uniqueness for full discretization]
\label{thm:existence_disc}
Let $m,N,r\in \N$ be fixed and let Assumptions~{\em AC} and {\em B} hold. 
Moreover, let $\lambda\tau \leq 1$.
Then, given $V_m$-valued and $\F_{0}$-measurable
random variables $u^0,v^0$ and right-hand side $f\in \mathcal{L}^q(V_A^*)$,
the fully discrete problem~\eqref{eq:2b} has
a unique solution $(v^n)_{n=1}^N$ in the sense that
if $(v_1^n)_{n=1}^N$ and $(v_2^n)_{n=1}^N$ both
satisfy~\eqref{eq:2b}  then
\[
\P\left(\max_{n=1,\ldots,N} |v_1^n - v_2^n| = 0\right) = 1.
\]
Furthermore, for all $n=1,\dots,N$, the $V_m$-valued random variables $v^n$ are $\mathcal{F}_{t_n}$-measurable.
\end{theorem}

\begin{proof}
We prove existence and uniqueness step by step.
Assume that the $V_m$-valued random variables $v^0, v^1, \ldots , v^{n-1}$
already satisfy~\eqref{eq:2b} (for all superscripts up to $n-1$).
Moreover, assume that $v^k$ is $\mathcal{F}_{t_k}$-measurable for $k=1,\dots,n-1$.
We will show that there is an $V_m$-valued and $\F_{t_n}$-measurable $v^n$
satisfying~\eqref{eq:2b}.

First recall that $u^k = u^0 + \tau\sum_{j=1}^k v^j$.
So $(u^k)_{k=0}^{n-1}$ is also known.
Recall that we are assuming that the dimension of $V_m$ is $m$.
Let $(\varphi_i)_{i=1}^m$ be a basis for $V_m$.
Then there is a one-to-one correspondence between any $w \in V_m$ and $\vec{w} = (w_1, \ldots, w_m)^T \in \R^m$ given by
$w = \sum_{i=1}^m w_i \varphi_i$.
We use this to define a norm on $\R^m$ by $\|\vec{w}\|_{\R^m} := \|w\|_{V_A}$.

Let $\Omega' \in \F_0$ be such that $\P(\Omega') = 1$ and such that, for all $\omega \in \Omega'$,
$t\mapsto \langle A(w+t z,\omega), v \rangle$ is continuous for any $w,z\in V_A$,
the joint monotonicity-like condition and the coercivity condition
on $A$ and $C$ are satisfied
and $B$ is linear, symmetric and strongly positive.
This is possible due to Assumptions AC and B.
For an arbitrary $\omega \in \Omega'$ and an arbitrary $v\in V_m$
and hence for some $\vec{v} = (v_1,\ldots,v_m)^T \in \R^m$,
define $\vec{h}: \Omega' \times \R^m \to \R^m$,
component-wise, for $l=1,\ldots,m$, as
\begin{equation*}
\begin{split}
h(\omega,\vec{v})_l :={} & \frac{1}{\tau}(v - v^{n-1}(\omega),\varphi_l) + \langle A(v,\omega), \varphi_l\rangle + \langle B(u^{n-1}(\omega) + \tau v, \omega),\varphi_l \rangle \\
& - \langle f^n(\omega), \varphi_l\rangle
- \left( C^r(u^{n-1}(\omega),v^{n-1}(\omega),\omega)) \frac{\Delta W^n(\omega)}{\tau},\varphi_l\right).
\end{split}
\end{equation*}
The first step in showing that~\eqref{eq:2b} has a solution is to show that
for each $\omega \in \Omega'$
there is some $\vec{v}$ such that $\vec{h}(\omega,\vec{v}) = \vec{0}$.
To that end, we would like to apply Lemma~\ref{lemma:consequenceofbrower}.
We see that
\begin{equation*}
\begin{split}
\vec{h}(\omega,\vec{v})\cdot \vec{v} ={} & \frac{1}{\tau}(v - v^{n-1}(\omega),v) + \langle A(v,\omega), v\rangle + \langle B(u^{n-1}(\omega) + \tau v,\omega),v \rangle \\
& - \langle f^n(\omega), v\rangle
- \left( C(u^{n-1}(\omega),v^{n-1}(\omega),\omega) \frac{\Delta W^{r,n}(\omega)}{\tau}, v\right).
\end{split}
\end{equation*}
Now we wish to find large $R(\omega) > 0$, which also depends on $m$,
such that if $\|v\|_{V_A} = R(\omega)$ then $\vec{h}(\omega,\vec{v})\cdot \vec{v} \geq 0$.
Note that since $V_A\hookrightarrow H$, we get
\begin{equation*}
(v-v^{n-1}(\omega),v) \geq |v|^2 - c|v^{n-1}(\omega)|\|v\|_{V_A}.
\end{equation*}
The coercivity in Assumption~{\em AC} together with Assumption~{\em B} imply
\begin{equation*}
\begin{split}
\vec{h}(\omega,\vec{v}) & \cdot \vec{v}  \geq{}   \frac{1}{\tau}(|v|^2 - c|v^{n-1}(\omega)|\|v\|_{V_A})  + \mu_A\|v\|_{V_A}^p 
+ \frac{1}{2}|C(0,v,\omega)|_{l^2(H)}^2 \\ 
& - \lambda_A|v|^2  
- \kappa - \|B(u^{n-1}(\omega),\omega)\|_{V_B^*}\|v\|_{V_B}
+ \tau \langle B(v,\omega), v\rangle \\
& - \|f^n(\omega)\|_{V_A^*}\|v\|_{V_A} 
- |C(u^{n-1}(\omega),v^{n-1}(\omega),\omega)|_{l^2(H)}
|v|\bigg|\frac{\Delta W^{r,n}(\omega)}{\tau}\bigg|.
\end{split}
\end{equation*}
Note that $V_m$ is finite dimensional and so there is $c_m > 0$ such that
$\|\varphi\|_{V_B} \leq c_m \|\varphi\|_{V_A}$ for all $\varphi \in V_m$.
Thus, noting also that 
$2\lambda_A\tau \le \lambda \tau \le 1$, we find that
\begin{equation*}
\begin{split}
\vec{h}(\omega,\vec{v}) & \cdot \vec{v}  \geq{}  \|v\|_{V_A}\bigg( \mu_A\|v\|_{V_A}^{p-1}
- c|v^{n-1}(\omega)| - c_m \|B(u^{n-1}(\omega),\omega)\|_{V_B^*}  \\
  & - \|f^n(\omega)\|_{V_A^*} - c |C(u^{n-1}(\omega),v^{n-1}(\omega),\omega)|_{l^2(H)}\bigg|\frac{\Delta W^{r,n}(\omega)}{\tau}\bigg| \bigg) - \kappa.
\end{split}
\end{equation*}
Now choose $R(\omega)$ large such that $R(\omega) \geq \kappa$ and also
\begin{equation*}
\begin{split}
& \mu_A R(\omega)^{p-1} - c|v^{n-1}(\omega)| - c_m \|B(u^{n-1}(\omega),\omega)\|_{V_B^*}
- \|f^n(\omega)\|_{V_A^*}
\\
& - c |C(u^{n-1}(\omega),v^{n-1}(\omega),\omega)|_{l^2(H)}
\bigg|\frac{\Delta W^{r,n}(\omega)}{\tau}\bigg| \geq 1.
\end{split}
\end{equation*}
Then, if $\|v\|_{V_A} = R(\omega)$, we have $h(\omega,\vec{v})\cdot \vec{v} \geq 0$.

Note that $\omega \in \Omega'$ and on this set we have linearity and boundedness
of $B$ and demicontinuity of $A$ (this follows from
the monotonicity-like assumption on $A$
and the hemicontinuity assumption on $A$). Thus the function $\vec{h}(\omega,\cdot)$ is continuous and
Lemma~\ref{lemma:consequenceofbrower} guarantees existence of $\vec{v}$
such that  $\vec{h}(\omega,\vec{v})=\vec{0}$.

Next we show that the zero of $\vec{h}(\omega,\cdot)$ is unique.
Assume that there are two distinct $\vec{v_1}$ and $\vec{v_2}$
such that $\vec{h}(\omega,\vec{v_1}) = \vec{0}$
and $\vec{h}(\omega,\vec{v_2}) = \vec{0}$.
Then
\begin{align*}
 0 = & \tau \left( \vec{h}(\omega, \vec{v}_1) - \vec{h}(\omega, \vec{v}_2) ,
 \vec{v}_1 -  \vec{v}_2 \right)
=
|v_1 - v_2|^2 \\
& + \tau \langle A(v_1,\omega) - A(v_2,\omega), v_1-v_2\rangle + \tau^2\langle B(v_1,\omega) - B(v_2,\omega) , v_1-v_2 \rangle .
\end{align*}
We recall that (\ref{eq:orig_monotonicity}) implies the monotonicity of $A + \lambda_A I$ and that $B$ is strongly positive. This yields
\begin{align*}
0 \ge |v_1 - v_2|^2 - \lambda_A \tau |v_1 - v_2|^2 + \mu_B \tau^2 \|v_1 - v_2\|_{V_B}^2 ,
\end{align*}
which shows that $v_1$ and $v_2$ cannot be distinct since $\lambda_A \tau \le 1/2$.
Hence the zero to $\vec{h}(\omega,\cdot)$ is unique.
Let $v^n(\omega) := v$ for $\omega \in \Omega'$ and $v^n(\omega) = 0$ for $\omega \in \Omega\setminus \Omega'$.
By Lemma~\ref{lemma:measurability}, we see that $v^n$ is $\F_{t_n}$-measurable.
\end{proof}

Now we need to obtain the a priori estimate.

\begin{theorem}[Discrete a~priori estimates]
\label{thm:disc-apriori}
Let $m,N,r\in \N$ be fixed and let Assumptions~{\em AC} and~{\em B} hold.
Moreover, for $f\in \mathcal{L}^q(V_A^*)$
let $(f_n)_{n=1}^N$ be given by (\ref{approx-fff}) and let $u^0$ and $v^0$ be $V_m$-valued and $\F_0$-measurable and such that $u^0 \in L^2(\Omega; H)$ and $v^0 \in L^2(\Omega; V_B)$.
Then for all $n = 1,\ldots, N$
\begin{equation}
\label{eq:apriori1}
\begin{split}
& \E\bigg[|v^n|^2  + |u^n|_B^2 + \sum_{j=1}^n |u^j - u^{j-1}|_B^2\bigg]\\
& \leq \E\bigg[|v^0|^2 + |u^0|_B^2
+ 2\tau \sum_{j=1}^n \langle f^j - Av^j,v^j\rangle  + \tau \sum_{j=1}^n|C^r(u^j,v^j)|_{l^2(H)}^2\bigg].
\end{split}
\end{equation}
Moreover, if $\lambda \tau < 1$ then
\begin{equation}
\label{eq:apriori2}
\begin{split}
& \E\bigg[|v^n|^2  +  |u^n|_B^2 +  \mu_A \tau \sum_{j=1}^n\|v^j\|_{V_A}^p +  \sum_{j=1}^n |u^j - u^{j-1}|_B^2\bigg]\\
& \leq c e^{\lambda T(1-\lambda \tau)^{-1}} \left( \E\bigg[|v^0|^2 + |u^0|_B^2\bigg] + \|f\|_{L^q((0,T)\times \Omega;V_A^*)}^q + T \right).
\end{split}
\end{equation}
\end{theorem}
\begin{proof}
By taking $\varphi = v^n$ in \eqref{eq:2bb} and using the relation
\begin{equation*}
(a-b , a) = \frac{1}{2}(|a|^2 - |b|^2 + |a-b|^2) ,
\end{equation*}
we get, for $j=1,\ldots,N$,
\begin{equation}
\label{eq:apriori_est_proof_1}
\begin{split}
& \frac{1}{2\tau}\big(|v^j|^2  -|v^{j-1}|^2  +|v^j-v^{j-1}|^2\big) + \langle Av^j + Bu^j, v^j\rangle \\
&  = \langle f^j, v^j \rangle + \bigg( C (u^{j-1}, v^{j-1})\frac{\Delta W^{r,j}}{\tau}, v^j \bigg).
\end{split}
\end{equation}
We note that $\langle Bu^j, v^j \rangle = (u^j,v^j)_B$ and so
\begin{equation*}
2\tau \sum_{j=1}^n (u^j,v^j)_B =
 2 \sum_{j=1}^n (u^j,u^j - u^{j-1})_B
= |u^n|_B^2 - |u^0|_B^2 + \sum_{j=1}^n |u^j - u^{j-1}|_B^2.
\end{equation*}
Thus, after multiplying by $2\tau$ and summing up from $j=1$ to $n$ in~\eqref{eq:apriori_est_proof_1}, we find
\begin{equation}
\label{eq:apriori_est_proof_2}
\begin{split}
&
|v^n|^2 + \sum_{j=1}^n|v^j - v^{j-1}|^2 + |u^n|_B^2 + \sum_{j=1}^n |u^j - u^{j-1}|_B^2 + 2\tau \sum_{j=1}^n \langle Av^j, v^j \rangle \\
& = |v^0|^2 + |u^0|_B^2 + 2\tau \sum_{j=1}^n \langle f^j ,v^j \rangle + 2 \sum_{j=1}^n ( C (u^{j-1}, v^{j-1})\Delta W^{r,j}, v^j ).
\end{split}
\end{equation}
Using Cauchy--Schwarz's and Young's inequalities, we obtain that
\begin{equation*}
\begin{split}
& ( C (u^{j-1}, v^{j-1})  \Delta W^{r,j}, v^j )\\
& = ( C (u^{j-1}, v^{j-1}) \Delta W^{r,j}, v^{j-1} ) + ( C (u^{j-1}, v^{j-1}) \Delta W^{r,j}, v^j - v^{j-1} ) \\
& \leq ( C (u^{j-1}, v^{j-1}) \Delta W^{r,j} , v^{j-1} ) + \frac{1}{2}|C (u^{j-1}, v^{j-1}) \Delta W^{r,j}|^2 + \frac{1}{2}|v^j - v^{j-1}|^2.
\end{split}
\end{equation*}
By the assumption on $(\F_t)$ and $W$, 
$\Delta W^{r,j}$ is independent of $\F_{t_{j-1}}$ and hence
\begin{equation*}
\E ( C(u^{j-1}, v^{j-1}) \Delta W^{r,j} , v^{j-1} ) = 0.
\end{equation*}
Furthermore, a straightforward calculation shows that
\begin{equation*}
\E |C (u^{j-1}, v^{j-1}) \Delta W^{r,j}|^2 =
\left\{
\begin{array}{ll}
0 &  \text{ if }  j = 1,\\
\tau \E |C^r (u^{j-1}, v^{j-1})|^2_{l^2(H)} & \,\,\textrm{if} \,\, j = 2,\ldots, N.
\end{array}
\right.
\end{equation*}
Using this and taking expectation in~\eqref{eq:apriori_est_proof_2} leads to
\begin{equation*}
\begin{split}
& \E\bigg[|v^n|^2 + |u^n|_B^2 + \sum_{j=1}^n |u^j - u^{j-1}|_2^B\bigg]\\
& \leq \E\bigg[|v^0|^2 + |u^0|_B^2
+ 2\tau \sum_{j=1}^n \langle f^j - Av^j,v^j\rangle  + \tau \sum_{j=2}^n|C^r(u^{j-1},v^{j-1})|_{l^2(H)}^2\bigg]
\end{split}
\end{equation*}
At this point, we only have to observe that
\begin{equation*}
\sum_{j=2}^n|C^r(u^{j-1},v^{j-1})|_{l^2(H)}^2 \leq \sum_{j=1}^n|C^r(u^j,v^j)|_{l^2(H)}^2
\end{equation*}
to obtain the first claim of the theorem.

Now we apply the coercivity condition in Assumption~{\em AC} and~\eqref{eq:mod_coercivity} to get, for any $j=1,\ldots, N$,
\begin{equation*}
-2\langle Av^j, v^j \rangle \leq - 2 \mu_A\|v^j\|_{V_A}^p - |C(u^j, v^j)|_{l^2(H)}^2 + \lambda |v^j|^2 + \lambda|u^j|_B^2 + \lambda.
\end{equation*}
Thus, again with Young's inequality, we find
\begin{equation*}
\begin{split}
& \E \bigg[|v^n|^2  + |u^n|_B^2 + \sum_{j=1}^n |u^j - u^{j-1}|_B^2 +  \mu_A \tau \sum_{j=1}^n \|v^j\|_{V_A}^p \bigg]
\\
&\leq \E\bigg[|v^0|^2 + |u^0|_B^2
 + c\tau \sum_{j=1}^n \|f^j\|_{V_A^*}^q + \lambda \tau\sum_{j=1}^n(1+|v^j|^2 + |u^j|_B^2)\bigg].
\end{split}
\end{equation*}
Then, since $\lambda \tau < 1$,
\begin{equation*}
\begin{split}
& \E \bigg[|v^n|^2  + |u^n|_B^2 + \sum_{j=1}^n |u^j - u^{j-1}|_B^2 + \mu_A \tau \sum_{j=1}^n \|v^j\|_{V_A}^p \bigg] \\
& \leq \frac{1}{1-\lambda \tau}\E\bigg[|v^0|^2 + |u^0|_B^2
+ c\tau \sum_{j=1}^n \|f^j\|_{V_A^*}^q + \lambda \tau\sum_{j=1}^{n-1}(|v^j|^2 + |u^j|_B^2) + \lambda T \bigg].
\end{split}
\end{equation*}
Since $f\in \mathcal{L}^q(V_A)$, we have
\begin{equation*}
\E\bigg[ \tau \sum_{j=1}^N\|f^j\|_{V_A^*}^q\bigg] \leq \E \int_0^T \|f(t)\|_{V_A^*}^q dt = \|f\|_{L^q((0,T)\times \Omega; V_A^*)}^q.
\end{equation*}
Finally, we can apply a discrete Gronwall lemma to obtain the second claim of the theorem and thus conclude the proof.
\end{proof}

\section{Weak limits from compactness}
\label{sec:weaklimits}
In this section, we consider a sequence of approximate problems (\ref{eq:scheme_continuous_formulation}) and
use compactness arguments and the a priori estimate of Theorem~\ref{thm:disc-apriori} to show that weak limits of the
piecewise-constant-in-time prolongations of the fully discrete approximate solutions exist and that they satisfy an equation closely resembling~\eqref{eq:1}.

Recall that we have constructed $v_\ell^-, v_\ell$ and $u_\ell^-, u_\ell$ in~\eqref{eq:ext1} and~\eqref{eq:ext2} by interpolating the solution of the fully discrete problem~\eqref{eq:2b}.
The following corollary is a direct consequence of  the a priori estimates of Theorem~\ref{thm:disc-apriori}.
\begin{corollary}
\label{corollary:to_apriori_est}
Let the assumptions of Theorem~\ref{thm:main} be fulfilled. Then
\begin{equation}
\label{eq:corollary_to_apriori_est_1}
\begin{split}
& \sup_{t\in [0,T]}\E|v_\ell^-(t)|^2 \leq c, \,\, \sup_{t\in [0,T]}\E|u_\ell^-(t)|_B^2 \leq c\,\,\textrm{ and }\,\,  \E\int_0^T \|v_\ell^-(t)\|_{V_A}^p dt \leq c,\\
& \sup_{t\in [0,T]}\E|v_\ell(t)|^2 \leq c, \,\, \sup_{t\in [0,T]}\E|u_\ell(t)|_B^2 \leq c\,\,\textrm{ and }\,\,  \E\int_0^T \|v_\ell(t)\|_{V_A}^p dt \leq c.
\end{split}
\end{equation}
Furthermore,
\begin{equation}
\label{eq:corollary_to_apriori_est_2}
\begin{split}
& \E\int_0^T \|Av_\ell^-(t)\|_{V_A^*}^q \, dt \leq c,\quad \E\int_0^T \|Av_\ell(t)\|_{V_A^*}^q \, dt \leq c,\\
&\E\int_0^T \|Bu_\ell(t)\|_{V_B^*}^2 \, dt \leq c,\\
& \E\int_0^T |C(u_\ell^-(t)),v_\ell^-(t)|_{l^2(H)}^2\, dt \leq c, \quad \E\int_0^T |C(u_\ell(t)),v_\ell(t)|_{l^2(H)}^2\, dt \leq c.  \\
\end{split}
\end{equation}
Finally,
\begin{equation}
\label{eq:corollary_to_apriori_est_3}
\E\int_0^T |u_\ell(t) - u_\ell^-(t)|_B^2 \,dt \leq c\tau_\ell .
\end{equation}
\end{corollary}

\begin{proof}
In view of the assumptions, the right-hand side of (\ref{eq:apriori2}) is uniformly bounded with respect to $\ell$.
This immediately implies~\eqref{eq:corollary_to_apriori_est_1}.
The assumptions on the growth of $A$ and $B$
together with \eqref{eq:Cbdd} and the first part of the corollary imply~\eqref{eq:corollary_to_apriori_est_2}.
Finally, \eqref{eq:corollary_to_apriori_est_3} is a consequence of~\eqref{eq:apriori2} and the observation that
\begin{equation*}
\E\int_0^T |u_\ell(t) - u_\ell^-(t)|_B^2\,dt = \tau_\ell \E\sum_{k=1}^{N_\ell}  |u^k - u^{k-1}|_B^2.
\end{equation*}
\end{proof}

We will need the following lemma to match the limits of the approximations $v_\ell$ of $v$ with their ``delayed'' and progressively measurable counterparts $v_\ell^-$,
see also Gy\"{o}ngy and Millet~\cite{gyongy:millet:on:discretization}.

\begin{lemma}
\label{lemma:indent_with_steklov}
Let $X$ be a separable and reflexive Banach space and let $\bar{p} \in (1,\infty)$.
Consider $\left( (x^n_\ell)_{n=0}^{N_\ell} \right)_{\ell\in \N}$ with $x^n_\ell \in L^{\bar{p}}(\Omega;X)$ for all $n = 0, 1, \ldots , N_\ell$ and $\ell \in \N$.
Consider the piecewise-constant-in-time processes $x_\ell$ and $x_\ell^-$ with $x_\ell(t_n) = x_\ell^-(t_n) = x^n_\ell$  and
\begin{equation*}
x_\ell(t) = x^n \,\,\textrm { if }\,\, t\in (t_{n-1},t_n) \,\, \textrm{ and } x_\ell^-(t) = x^{n-1} \,\,\textrm { if }\,\, t\in (t_{n-1},t_n)
\end{equation*}
for $n=1,\ldots,N_\ell$, $\ell\in \N$.
Assume that $(x_\ell)_{\ell\in\N}$ and $(x_\ell^-)_{\ell\in\N}$ are bounded in  $L^{\bar{p}}((0,T)\times\Omega; X)$.
Then there is a subsequence denoted by $\ell'$  and $x , x^- \in L^{\bar{p}}((0,T)\times\Omega; X)$ such that $x_{\ell'} \rightharpoonup {x}$
and $x_{\ell'}^- \rightharpoonup x^-$ in $L^{\bar{p}}((0,T)\times\Omega; X)$
as $\ell' \to \infty$ with $x = x^-$.
\end{lemma}

\begin{proof}
The existence of a subsequence and of $x, x^- \in L^{\bar{p}}((0,T)\times\Omega; X)$ such that
$x_{\ell'} \rightharpoonup {x}$
and $x_{\ell'}^- \rightharpoonup x^-$ in $L^{\bar{p}}((0,T)\times\Omega; X)$
as $\ell' \to \infty$
follows from standard compactness arguments since $L^{\bar{p}}((0,T)\times\Omega; X)$ is reflexive. It remains to show that
$x = x^-$.

To that end, we will employ the averaging operator $S_\ell:L^{\bar{q}}((0,T)\times \Omega;X^*) \to L^{\bar{q}}((0,T)\times \Omega;X^*)$ ($1/\bar{p} + 1/\bar{q} = 1$) defined by
\begin{equation*}
(S_\ell y)(t) :=
\left\{
\begin{array}{ll}
\displaystyle
\frac{1}{\tau_\ell}\int_{\theta_\ell^+(t)}^{\theta_\ell^+(t+\tau_\ell)} y(s)ds &  \ \textrm{ if } \  t \in [0,T-\tau_\ell] , \\
0 &  \ \textrm { otherwise. }
\end{array}
\right.
\end{equation*}
It can be shown for all $y \in L^{\bar{q}}((0,T)\times \Omega;X^*)$, using standard arguments, that $S_\ell y \to y$ in $L^{\bar{q}}((0,T)\times \Omega;X^*)$ as $\ell \to \infty$.

Let $y\in L^{\bar{q}}((0,T)\times \Omega;X^*)$.
A short calculation then reveals that
\begin{equation}
\label{eq:calc_with_S}
\int_0^T \langle(S_\ell y)(t), x_\ell(t)\rangle dt = \int_{\tau_\ell}^T \langle y(t),x_\ell^-(t)\rangle dt
\end{equation}
and hence
\begin{equation*}
\begin{split}
& \E\int_0^T \langle y(t), x(t) - x^-(t)\rangle \, dt
=  \E\int_0^T \langle y(t), x(t) - x_{\ell'}^-(t)\rangle \,dt \\
& + \E\int_0^T \langle y(t),x_{\ell'}^-(t) - x_{\ell'}(t)\rangle\, dt
  + \E\int_0^T \langle y(t),x_{\ell'}(t) - x^-(t)\rangle\,dt.
\end{split}
\end{equation*}
The first and last integral on the right-hand side converge to $0$ as $\ell'\to \infty$.
We observe that due to~\eqref{eq:calc_with_S}
\begin{equation*}
\begin{split}
\E\int_0^T \langle y(t),x_{\ell'}^-(t) - x_{\ell'}(t)\rangle\,dt ={}& \E\int_0^{\tau_\ell} \langle y(t), x_{\ell'}^-(t)\rangle\, dt \\
& + \E \int_0^T \langle(S_{\ell'}y)(t) - y(t),x_{\ell'}(t)\rangle\, dt.
\end{split}
\end{equation*}
The first integral on the right-hand side converges to $0$ since $\tau_\ell \to 0$ and since $(x_{\ell'}^-)_{\ell\in\N}$ is bounded in $L^{\bar{p}}((0,T)\times \Omega;X)$.
The second integral on the right-hand side converges to $0$ since
$S_{\ell'} y \to y$ in $L^{\bar{q}}((0,T)\times \Omega;X^*)$ as $\ell' \to \infty$ and since
$(x_{\ell'})_{\ell\in\N}$ is bounded in $L^{\bar{p}}((0,T)\times \Omega;X)$.
This finally shows that $x = x^-$ in $L^{\bar{p}}((0,T)\times \Omega;X)$.
\end{proof}

\begin{lemma}
\label{lemma:weaklimits}
Let the assumptions of Theorem~\ref{thm:main} be fulfilled.
Then there is a subsequence denoted by $\ell'$ such that:
\begin{enumerate}[(i)]
\item There is $v\in \mathcal{L}^p(V_A)$ such that $v_{\ell'}^- \rightharpoonup v$ and $v_{\ell'} \rightharpoonup v$ in $L^p((0,T)\times \Omega; V_A)$.
There is $\xi \in L^2(\Omega;H)$ such that $v_{\ell'}^-(T) = v_{\ell'}(T) \rightharpoonup \xi$ in $L^2(\Omega;H)$ as $\ell' \to \infty$.
\item
There is $u\in \mathcal{L}^2(V_B)$ such that $u_{\ell'}^- \rightharpoonup u$
and $u_{\ell'} \rightharpoonup u$ in $L^2((0,T)\times \Omega;V_B)$
as $\ell'\to \infty$.
Furthermore, $u - u_0 = Kv$ in $\mathcal{L}^p(V_A)$ and
 the paths of $u-u_0$ are absolutely continuous. Finally, $u_{\ell'}^-(T) = u_{\ell'}(T)\rightharpoonup u(T)$ in $L^2(\Omega; V_B)$ and $u(0) = u_0$.
\item There is $a \in \mathcal{L}^q(V_A^*)$ such that $Av_{\ell'} \rightharpoonup a$ in $L^q((0,T)\times \Omega; V_A^*)$.
There is $\bar{c} \in \mathcal{L}^2(l^2(H))$ such that 
$C^{r_{\ell'}}(u_{\ell'}^-,v_{\ell'}^-)$, $C(u_{\ell'},v_{\ell'})$
and $C^{r_{\ell'}}(u_{\ell'},v_{\ell'})$
all converge weakly to $\bar{c}$ in $L^2((0,T)\times \Omega; l^2(H))$ as $\ell' \to \infty$.
\end{enumerate}
\end{lemma}
\begin{proof}
We begin by observing that $L^p((0,T)\times\Omega;V_A)$, 
$\mathcal{L}(V_A)$ and $L^2(\Omega;H)$ are reflexive.
Then, due to Corollary~\ref{corollary:to_apriori_est} 
and due to e.g. Br\'ezis~\cite[Theorem 3.18]{brezis:functional}, there  are 
$v \in L^p((0,T)\times \Omega;V_A)$
${v}^- \in \mathcal{L}(V_A)$
and $\xi \in L^2(\Omega; H)$ and a subsequence denoted by $\ell'$ such that
$v_{\ell'}^- \rightharpoonup v^-$ and $v_{\ell'} \rightharpoonup {v}$  in $L^p((0,T)\times \Omega; V_A)$
as well as $v_{\ell'}(T) \rightharpoonup \xi$ in $L^2(\Omega; H)$ as $\ell'\to\infty$. 
To complete the proof of the first statement, we simply need to
apply Lemma~\ref{lemma:indent_with_steklov} to see that $v=v^-$.

Using the same argument as in the first part of the proof, we obtain
$u_{\ell'}^-\rightharpoonup u$
and $u_{\ell'}\rightharpoonup u$
in $L^2((0,T)\times \Omega;V_B)$ with $u\in \mathcal{L}^2(V_B)$ as well as
$u_{\ell'}(T) \rightharpoonup \eta$ with $\eta \in L^2(\Omega, V_B)$  as $\ell'\to\infty$. By the way, \eqref{eq:corollary_to_apriori_est_3} implies that
\begin{equation*}
\|u_\ell - u_\ell^-\|_{L^2((0,T)\times \Omega; V_B)} \to 0 \text{ as } \ell \to \infty ,
\end{equation*}
which also shows that the weak limits of $u_\ell$ and $u_\ell^-$ coincide.

Now we would like to show that $u-u_0 = Kv$. A straightforward calculation shows that
\begin{equation*}
u_{\ell} - u^0_{\ell} = Kv_{\ell} + e_{\ell} , \text{ where }
e_\ell(t):= \int_t^{\theta_\ell^+(t)}v_\ell(s)ds.
\end{equation*}
Another straightforward calculation also shows that
$Kv_{\ell'} \rightharpoonup Kv$ in $L^p((0,T)\times \Omega;V_A)$ since $v_{\ell'} \rightharpoonup v$ in $L^p((0,T)\times \Omega;V_A)$ as $\ell'\to\infty$.
Due to Theorem~\ref{thm:disc-apriori}, we have
\begin{equation*}
\begin{split}
\|e_\ell\|_{{L}^p((0,T)\times \Omega ; V_A)}^p
 & = \E\int_0^T \bigg\|\int_t^{\theta_\ell^+(t)}v_\ell(s)ds \bigg\|_{V_A}^p dt\\
& = \E\sum_{j=1}^{N_\ell} \int_{t_{j-1}}^{t_j} (t_j - t)^p \|v^{j}\|_{V_A}^p dt\\
& \le \tau_\ell^p \E \tau_\ell\sum_{j=1}^{N_\ell}\|v^{j}\|_{V_A}^p
\leq c\tau_\ell^p \to 0 \text{ as } \ell \to \infty.
\end{split}
\end{equation*}
It follows that
\begin{equation*}
u_{\ell'} - u^0_{\ell'} = Kv_{\ell'} + e_{\ell'}\rightharpoonup Kv
\end{equation*}
in $L^p((0,T)\times \Omega;V_A)$ as $\ell'\to\infty$, which shows that $u-u_0 = Kv$ in view of $u_{\ell'} \rightharpoonup u$ in $L^2((0,T)\times \Omega;V_B)$ as $\ell'\to\infty$ and $u^0_\ell \to u_0$ in $L^2(\Omega;V_B)$ as $\ell\to\infty$.

Hence almost all paths of $u-u_0$ are absolutely continuous as functions mapping $[0,T]$ into $V_A$. Moreover, $u(0) = u_0$ since $(Kv)(0) = 0$.

To complete the proof of the second statement of the lemma, we have to show that $\eta = u(T)$. Again, a straightforward calculation shows that
$(Kv_{\ell'})(T) \rightharpoonup (Kv)(T)$ in $L^p(\Omega;V_A)$
as $\ell'\to\infty$ since for all $g \in L^q(\Omega;V_A^*)$
\begin{equation*}
\E \,\langle g, (Kv_{\ell'})(T) - (Kv)(T)\rangle = \E \int_0^T \langle  g ,
v_{\ell'}(t) - v(t) \rangle dt
\end{equation*}
and since $v_{\ell'} \rightharpoonup v$ in $L^p((0,T)\times \Omega;V_A)$ as $\ell'\to\infty$.
Therefore, we find that $\eta - u_0 = (Kv)(T) = u(T) - u_0$.

The second part of Corollary~\ref{corollary:to_apriori_est} 
(see (\ref{eq:corollary_to_apriori_est_2})) implies (iii) with the same arguments as before. 
In particular, the weak limits of $Av_{\ell'}^-$ and of $C^{r_{\ell'}}(u_{\ell'}^-, v_{\ell'}^-)$ are progressively measurable and thus
$a\in \mathcal{L}^q(V_A^*)$
as well as $\bar{c}\in \mathcal{L}^2(l^2(H))$.
Indeed,~\eqref{eq:corollary_to_apriori_est_2} implies that
\[
\sum_{j=r_{\ell'}}^\infty \E\int_0^T |C_j(u_{\ell'}, v_{\ell'})|^2 dt \to 0
\]
as $\ell'\to \infty$. 
This in turn implies that 
\[
\|C^{r_{\ell'}}(u_{\ell'}, v_{\ell'})-C(u_{\ell'}, v_{\ell'})\|_{L^2((0,T)\times \Omega;l^2(H))} \to 0.
\]
Using this observation allows us to show that the weak limits of $C^{r_{\ell'}}(u_{\ell'}, v_{\ell'})$ 
and $C(u_{\ell'}, v_{\ell'})$ coincide in $L^2((0,T)\times \Omega;l^2(H))$.
Moreover, due to Lemma~\ref{lemma:indent_with_steklov}, the weak limits of $C^{r_{\ell'}}(u_{\ell'}, v_{\ell'})$ and 
$C^{r_{\ell'}}(u_{\ell'}^-, v_{\ell'}^-)$ also coincide.
\end{proof}

At this point, we are ready to take the limit in~\eqref{eq:scheme_continuous_formulation} along $\ell'\to\infty$.

\begin{lemma}
\label{lemma:eqns}
Let the assumptions of Theorem~\ref{thm:main} be fulfilled. Then for $(dt\times d\P)$-almost all $(t,\omega) \in (0,T)\times \Omega$
\begin{equation}
\label{eq:5a}
  v(t) + \int_0^t a(s)ds + \int_0^t Bu(s) ds = v_0 + \int_0^t f(s)ds + \int_0^t \bar{c}(s)dW(s) \text{ in } V_A^* ,
\end{equation}
and there is an $H$-valued continuous modification of $v$ (which we denote by $v$ again) such that for all $t\in [0,T]$
\begin{equation}
\label{eq:7}
\begin{split}
|v(t)|^2 + |u(t)|_B^2 ={} & |v_0|^2 + |u_0|_B^2 + \int_0^t \big[ 2\langle f(s)-a(s), v(s)\rangle + |\bar{c}(s)|^2\big]ds\\
& + 2\int_0^t (v(s),\bar{c}(s)dW(s)).
\end{split}
\end{equation}
Finally, $\xi = v(T)$ and thus $v_{\ell'}(T) \rightharpoonup v(T)$ in $L^2(\Omega; H)$ as $\ell'\to\infty$.
\end{lemma}

\begin{proof}
In what follows, we only write $\ell$ instead of $\ell'$.
Let us fix $m \leq m_\ell$ and take $\varphi = \psi(t)\bar{\varphi}$ in
\eqref{eq:scheme_continuous_formulation} with $\bar{\varphi} \in V_m$ and $\psi \in L^p((0,T)\times \Omega; \R)$.
Integrating from $0$ to $T$ and taking the expectation then leads to
\begin{equation*}
\begin{split}
& \E\int_0^T \bigg[(v_\ell(t),\varphi(t)) + \bigg\langle\int_0^{\theta_\ell^+(t)} (Av_\ell(s)+Bu_\ell(s))ds,\varphi(t) \bigg\rangle \bigg]dt \\
& = \E\int_0^T \bigg[ (v^0_\ell,\varphi(t)) + \bigg\langle\int_0^{\theta_\ell^+(t)}f_\ell(s)ds,\varphi(t) \bigg\rangle \\
& \quad + \bigg(\int_{\tau_\ell}^{\theta_\ell^+(t)} C^{r_\ell}(u_\ell^-(s),v_\ell^-(s))dW(s),\varphi(t)\bigg)\bigg]dt.
\end{split}
\end{equation*}
We subsequently see that
\begin{equation}
\label{eq:6}
\begin{split}
& \E\int_0^T \bigg[(v_\ell(t),\varphi(t)) + \langle (KAv_\ell)(t),\varphi(t)\rangle + \langle (K Bu_\ell)(t),\varphi(t)\rangle \bigg] dt \\
& = \E\int_0^T\bigg[ (v^0_\ell,\varphi(t)) + \langle (K f_\ell)(t),\varphi(t)\rangle \\
& \quad + \bigg(\int_0^t C^{r_\ell}(u_\ell^-(s),v_\ell^-(s))dW(s),\varphi(t)\bigg)\bigg]dt + R_\ell^1 + R_\ell^2 + R_\ell^3,
\end{split}
\end{equation}
where
\begin{align*}
R_\ell^1 &:= \E\int_0^T  \bigg\langle\int_t^{\theta_\ell^+(t)} (f_\ell(s) - Av_\ell(s) - Bu_\ell(s))ds,\varphi(t) \bigg\rangle dt,
\\
R_\ell^2  &:= \E\int_0^T \bigg(\int_0^{\tau_\ell} C(u_\ell^-(s),v_\ell^-(s))dW^{r_\ell}(s),\varphi(t)\bigg)dt,
\\
R_\ell^3  &:= \E \int_0^T \bigg(\int_t^{\theta_\ell^+(t)} C(u_\ell^-(s),v_\ell^-(s))dW^{r_\ell}(s),\varphi(t)\bigg)dt.
\end{align*}
We will now show that $R_\ell^1, \, R_\ell^2, \, R_\ell^3 \to 0$ as $\ell \to \infty$.

Because of
\begin{align*}
R_\ell^1
&= \E \sum_{j=1}^{N_\ell} \int_{t_{j-1}}^{t_j}
\bigg\langle\int_t^{t_j} (f^j - Av^j - Bu^j)ds,\varphi(t) \bigg\rangle dt
\\
&= \E \int_0^T (\theta_\ell^+(t) - t)\left\langle f_\ell(t) - Av_\ell(t) - Bu_\ell(t),\varphi(t) \right\rangle dt ,
\end{align*}
we obtain, using H\"older's inequality and
Corollary~\ref{corollary:to_apriori_est},
\begin{align*}
|R_\ell^1|  \le{} &
\tau_\ell \E\int_0^T \left|\left\langle f_\ell(t) - Av_\ell(t) - Bu_\ell(t),\varphi(t)\right\rangle \right| dt\\
 \leq{} & \tau_\ell \big(
\big(\|f_\ell\|_{L^q((0,T)\times \Omega; V_A^*)}
+ \|Av_\ell\|_{L^q((0,T)\times \Omega; V_A^*)}\big)\|\varphi\|_{L^p((0,T)\times \Omega; V_A)}\\
& + \|Bu_\ell\|_{L^2((0,T)\times \Omega; V_B^*)} \|\varphi\|_{L^2((0,T)\times \Omega; V_B)} \big) \to 0
\end{align*}
as $\ell\to\infty$.
Using H\"older's inequality and It\^o's isometry (see, e.g., Pr\'ev\^ot and R\"ockner~\cite[Section 2.3]{prevot:rockner:concise}), we find
with $u_\ell^-(t) = u^0_\ell$ and $v_\ell^-(t) = 0$ if $t\in [0,\tau_\ell)$
that
\begin{align*}
|R_\ell^2| & \leq \E\int_0^T\bigg|\int_0^{\tau_\ell} C(u_\ell^-(s),v_\ell^-(s))dW^{r_\ell}(s)\bigg||\varphi(t)|dt\\
& \leq
\left(\E \int_0^T \left| \int_0^{\tau_\ell} C(u^0_\ell,0)dW^{r_\ell}(s) \right|^2dt\right)^{1/2} \|\varphi\|_{L^2((0,T)\times \Omega; H)}\\
& =  \left(\E \int_0^T  \int_0^{\tau_\ell} |C(u^0_\ell,0)|_{l^2(H)}^2 ds dt\right)^{1/2} \|\varphi\|_{L^2((0,T)\times \Omega; H)}\\
& =  ( \tau_\ell T )^{1/2}
\left(\E |C(u^0_\ell,0)|_{l^2(H)}^2 \right)^{1/2} \|\varphi\|_{L^2((0,T)\times \Omega; H)}
 \to 0
\end{align*}
as $\ell\to\infty$.
Similarly, using also Corollary~\ref{corollary:to_apriori_est}, we see that
\begin{align*}
|R_\ell^3| & \le \left( \E \int_0^T \bigg|\int_t^{\theta_\ell^+(t)}
C(u_\ell^-(s),v_\ell^-(s))dW^{r_\ell}(s)
\bigg|^2 dt  \right)^{1/2}
\|\varphi\|_{L^2((0,T)\times \Omega; H)}
\\
& = \left( \E \int_0^T\int_t^{\theta_\ell^+(t)} \bigg|
C(u_\ell^-(s),v_\ell^-(s))\bigg|_{l^2(H)}^2 dsdt  \right)^{1/2}
\|\varphi\|_{L^2((0,T)\times \Omega; H)}
\\
& = \left( \E \int_0^T (\theta_\ell^+(t) - t ) \bigg|
C(u_\ell^-(t),v_\ell^-(t))\bigg|_{l^2(H)}^2 dt  \right)^{1/2}
\|\varphi\|_{L^2((0,T)\times \Omega; H)}
\\
& \le \tau_\ell^{1/2} \left( \E \int_0^T  \bigg|
C(u_\ell^-(t),v_\ell^-(t))\bigg|_{l^2(H)}^2 dt  \right)^{1/2}
\|\varphi\|_{L^2((0,T)\times \Omega; H)}
\to 0
\end{align*}
as $\ell\to\infty$.

We would now like to let $\ell\to \infty$ in~\eqref{eq:6}.
A simple calculation shows that $KAv_\ell \rightharpoonup Ka$ in $L^q((0,T)\times \Omega ; V_A^*)$ as $\ell \to \infty$ since $Av_\ell \rightharpoonup a$ in $L^q((0,T)\times \Omega ; V_A^*)$ as $\ell \to \infty$. Analogously, we observe that
$KBu_\ell \rightharpoonup KBu$ in $L^2((0,T)\times \Omega ; V_B^*)$ as $\ell \to \infty$ since $u_\ell \rightharpoonup u$ in $L^2((0,T)\times \Omega ; V_B)$ and thus $Bu_\ell \rightharpoonup Bu$ in $L^2((0,T)\times \Omega ; V_B^*)$ as $\ell \to \infty$ (note that $B$ is linear bounded and thus weakly-weakly continuous).

The stochastic integral is a linear bounded operator mapping $\mathcal{L}^2(l^2(H))$
into $\mathcal{L}^2(H)$.
Indeed, by It\^o's isometry (see again Pr\'ev\^ot and R\"ockner~\cite[Section~2.3]{prevot:rockner:concise}), we have for any $g\in \mathcal{L}^2(l^2(H))$
\begin{equation*}
\begin{split}
\bigg\|\int_0^\cdot g(s)dW(s)\bigg\|_{L^2((0,T)\times \Omega;H)}^2 
& = \E \int_0^T \int_0^t |g(s)|_{l^2(H)}^2\, dsdt\\ 
& \leq T\|g\|_{L^2((0,T)\times \Omega;l^2(H))}^2.	
\end{split}
\end{equation*}
Hence the stochastic integral maps weakly convergent sequences in
$\mathcal{L}^2(l^2(H))$ into weakly convergent sequences in
$\mathcal{L}^2(H)$.
With Lemma~\ref{lemma:weaklimits}, we thus obtain
\begin{equation*}
\E\int_0^T \!\! \bigg(\int_0^t C^{r_\ell}(u_{\ell}^-(s),v_{\ell}^-(s))dW(s),\varphi(t)\bigg)dt \to
\E\int_0^T \!\!\bigg(\int_0^t \bar{c}(s)dW(s),\varphi(t)\bigg)dt
\end{equation*}
as $\ell \to \infty$.

So, taking the limit in~\eqref{eq:6} as $\ell\to \infty$ and using also
$v_\ell \rightharpoonup v$ in $L^2((0,T)\times \Omega;H)$,
$v^0_\ell \to v_0$ in $L^2(\Omega;H)$ and $f_\ell \to f$ in $L^q((0,T)\times\Omega;V_A^*)$
as $\ell\to\infty$ (the latter can be shown by standard arguments), we arrive at
\begin{equation*}
\begin{split}
& \E \int_0^T \bigg[ (v(t), \varphi(t)) + \left\langle \int_0^t  a(s)ds, \varphi(t)\right\rangle + \left\langle \int_0^t Bu(s)ds, \varphi(t)\right\rangle  \bigg] dt\\
& =  \E \int_0^T\bigg[ (v_0, \varphi(t)) + \left\langle \int_0^t  f(s)ds, \varphi(t)\right\rangle
+ \left(\int_0^t  \bar{c}(s)dW(s),\varphi(t)\right)\bigg]dt ,
\end{split}
\end{equation*}
which holds for all $\varphi = \psi \bar{\varphi}$ with
$\psi \in L^p((0,T)\times \Omega; \R)$ and
$\bar{\varphi} \in V_m$. As $(V_m)_{m\in \N}$ is a Galerkin scheme for $V_A$, the above equation indeed holds for
$\varphi = \psi \bar{\varphi}$ with any  $\bar{\varphi} \in V_A \hookrightarrow V_B$.
This proves \eqref{eq:5a}.

Now we need to use $V_A\hookrightarrow V_B$.
With this assumption, we can apply the It\^o formula for the square of the norm
(see, e.g., Krylov and Rozovskii~\cite[Theorem~3.1 and
Section~2]{krylov:rozovskii:stochastic} or
Pr\'ev\^ot and R\"ockner~\cite[Theorem~4.2.5]{prevot:rockner:concise}).
Thus we conclude that $v$ has an $H$-valued continuous modification (which we label $v$ again) such that (\ref{eq:5a}) holds for all $t\in [0,T]$ and
\begin{equation*}
\begin{split}
|v(t)|^2 - |v_0|^2  ={}&  \int_0^t \big[ 2\langle f(s)-a(s)-Bu(s), v(s)\rangle + |c(s)|^2\big]ds\\
& + 2\int_0^t (v(s),c(s)dW(s)) .
\end{split}
\end{equation*}
With
\begin{align*}
\int_0^t \langle Bu(s) & , v(s)\rangle ds =
\int_0^t \langle B (u_0 + (Kv)(s)), v(s)\rangle ds
\\
&= \langle Bu_0 , (Kv)(t) \rangle + \int_0^t \int_0^s \langle Bv(\sigma) , v(s) \rangle d\sigma ds
\\
&= \langle Bu_0 , (Kv)(t) \rangle + \int_0^t \int_\sigma^t \langle Bv(\sigma) , v(s) \rangle ds d\sigma
\\
&=
\langle Bu_0 , (Kv)(t) \rangle + \langle B(Kv)(t) , (Kv)(t) \rangle
- \int_0^t \langle B v(\sigma) , (Kv)(\sigma) \rangle d\sigma
\\
& =  \langle B (u(t) + u_0) , (u(t)-u_0) \rangle
- \int_0^t \langle B u(s) , v(s) \rangle ds
\end{align*}
and thus
\begin{equation}\label{eq:Buv}
2\int_0^t \langle Bu(s), v(s)\rangle ds = |u(t)|_B^2 - |u_0|_B^2 ,
\end{equation}
we arrive at \eqref{eq:7}.

Recall that $\xi$ is the weak limit of $v_{\ell}(T)$ in $L^2(\Omega;H)$.
Using a similar limiting argument as above, we obtain 
that
\begin{equation*}
\xi + \int_0^T a(s)\, ds + \int_0^T Bu(s)\, ds 
= v_0 + \int_0^T f(s)\, ds + \int_0^T \bar{c}(s)\,dW(s) 	
\end{equation*}
with the equality holding almost surely in $H$.
This, together with the knowledge that 
$v$ has an $H$-valued continuous modification 
and with~\eqref{eq:5a}, implies that $\xi=v(T)$.
\end{proof}

\section{Identifying the limits in the nonlinear terms. Proof of convergence and existence}
\label{sec:identlims}

In this section, we continue the considerations of the previous section and we will use a variant of a well known monotonicity argument to identify $a$ with $Av$ and $c$ with $C(u,v)$.
This will conclude the proof of the main theorem of the paper.
We will need the following observation.
\begin{lemma}
\label{lemma:int_by_pts_cals}
Let $a$ and $b$ be real-valued integrable functions such that for all $t \in [0,T]$
\begin{equation}
\label{eq:int_by_pts_cal1}
a(t) \leq a(0) + \int_0^t b(s)ds.	
\end{equation}
Then for all $\kappa \ge 0$ and for all $t \in [0,T]$
\begin{equation}
\label{eq:int_by_pts_cal2}
e^{-\kappa t} a(t) + \kappa \int_0^t e^{-\kappa s} a(s) ds
\leq a(0) + \int_0^t e^{-\kappa s} b(s) ds.
\end{equation}
Moreover, if equality holds in~\eqref{eq:int_by_pts_cal1}
then equality also holds in~\eqref{eq:int_by_pts_cal2}.
\end{lemma}
\begin{proof}
Using the assumption and integrating by parts, we find
\[
\begin{split}
& e^{-\kappa t}a(t) + \int_0^t \kappa e^{-\kappa s} a(s) ds
\leq e^{-\kappa t} a(0) + e^{-\kappa t} \int_0^t b(s) ds \\
& + \int_0^t \kappa e^{-\kappa s} \bigg[a(0) + \int_0^s b(u)du \bigg] ds
 = a(0) + \int_0^t e^{-\kappa s} b(s) ds.
\end{split}
\]	
This proves the assertion.
\end{proof}

\begin{proof}[Proof of Theorem~\ref{thm:main}]
Let
\begin{equation*}
\varphi_\ell(t) :=
\left\{
\begin{array}{lcl}
\displaystyle
\E ( |v_\ell(t)|^2 + |u_\ell(t)|_B^2 ) & \textrm{ if } & t \in (0,T],
\\[1ex]
\displaystyle
\E ( |v^0_\ell|^2 + |u^0_\ell|_B^2 ) & \textrm{ if } & t = 0.
\end{array}
\right.
\end{equation*}
Then from Theorem~\ref{thm:disc-apriori}, in particular~\eqref{eq:apriori1},
we find for all $t\in [0,T]$
\begin{equation*}
\varphi_\ell(t) \! \leq \varphi_\ell(0) + \E\int_0^t \!\!\big[ 2\langle f_\ell(s) - Av_\ell(s), v_\ell(s) \rangle + |C^{r_\ell}(u_\ell(s),v_\ell(s))|_{l^2(H)}^2 \big]ds + R_\ell(t),
\end{equation*}
where
\begin{equation*}
R_\ell(t) := \E \int_t^{\theta_\ell^+(t)} \big[ 2\langle f_\ell(s) - Av_\ell(s), v_\ell(s) \rangle + |C^{r_\ell}(u_\ell(s),v_\ell(s))|_{l^2(H)}^2 \big]ds.
\end{equation*}
Note that $R_\ell(0) = R_\ell(T) = 0$.
From Lemma~\ref{lemma:int_by_pts_cals}, we see that
\begin{equation}
\label{eq:int_by_parts_in_disc_eq}
\begin{split}
& e^{-\lambda T}  \varphi_\ell(T) \leq \varphi_\ell(0) - \lambda \int_0^T e^{-\lambda s} \varphi_\ell(s)ds \\
& + \E\int_0^T \!\!\!\! e^{-\lambda s}\big[ 2\langle f_\ell(s) - Av_\ell(s), v_\ell(s) \rangle + |C^{r_\ell}(u_\ell(s),v_\ell(s))|_{l^2(H)}^2 \big]ds + \bar{R}_\ell,
\end{split}
\end{equation}
where $\bar{R}_\ell := \lambda \int_0^T e^{-\lambda s} |R_\ell(s)| ds$.
We will show that $\bar{R}_\ell \to 0$ as $\ell \to \infty$.
Indeed,
\begin{equation*}
\begin{split}
& \bar{R}_\ell
 \leq \lambda \E \int_0^T  \int_t^{\theta_\ell^+(t)} \big| 2\langle f_\ell(s) - Av_\ell(s), v_\ell(s) \rangle + |C(u_\ell(s),v_\ell(s))|_{l^2(H)}^2 \big| ds dt\\
& \leq c \tau_\ell \E \int_0^T \big[ 2\left(\|f_\ell(t)\|_{V_A^*} + \|Av_\ell(t)\|_{V_A^*}\right)\|v_\ell(t)\|_{V_A} + |C(u_\ell(t), v_\ell(t))|_{l^2(H)}^2 \big] dt\\
& \leq c \tau_\ell,
\end{split}
\end{equation*}
since the integrand is piecewise constant in time and since we can apply
Young's inequality and Corollary~\ref{corollary:to_apriori_est}.

Now we are ready to apply the monotonicity-like assumption~\eqref{eq:mod_monotonicity}.
Let $w \in \mathcal{L}^p(V_A)$ and let $z \in \mathcal{L}^2(V_B)$.
We see that
\begin{equation*}
\begin{split}
& \E\int_0^T  e^{-\lambda s} \langle Av_\ell(s), v_\ell(s) \rangle ds
= \E \int_0^T  e^{-\lambda s} \langle Av_\ell(s) - Aw(s), v_\ell(s) - w(s) \rangle ds\\
& \, +  \E\int_0^T e^{-\lambda s} [ \langle Aw(s), v_\ell(s) - w(s)\rangle +  \langle Av_\ell(s),w(s)\rangle ] ds\\
& \geq  \frac{1}{2}\E \int_0^T \!\! e^{-\lambda s} \big[|C(u_\ell(s),v_\ell(s)) - C(z(s),w(s))|_{l^2(H)}^2 \\
& \, - \lambda|v_\ell(s)-w(s)|^2 - \lambda |u_\ell(s) - z(s)|_B^2\big]ds\\
& \, +  \E\int_0^T e^{-\lambda s} [ \langle Aw(s), v_\ell(s) - w(s)\rangle +  \langle Av_\ell(s),w(s)\rangle ] ds.
\end{split}
\end{equation*}
Then from~\eqref{eq:int_by_parts_in_disc_eq}, we can deduce that
\begin{equation}
\label{eq:mono_trick_0}
\begin{split}
& e^{-\lambda T}  \E\big(|v_\ell(T)|^2 + |u_\ell(T)|_B^2\big)\\
& \leq \E\big(|v^0_\ell|^2 + |u^0_\ell|_B^2\big) - \lambda \int_0^T e^{-\lambda s} \E\big(|v_\ell(s)|^2 + |u_\ell(s)|_B^2\big)ds \\
& \, + \E\int_0^T \!\!\! e^{-\lambda s}\big[ 2\langle f_\ell(s) - Av_\ell(s), v_\ell(s) \rangle + |C(u_\ell(s),v_\ell(s))|_{l^2(H)}^2 \big]ds + \bar{R}_\ell
\\
& \leq   \E\big(|v^0_\ell|^2 + |u^0_\ell|_B^2\big)  + 2\E\int_0^T e^{-\lambda s}\langle f_\ell(s),v_\ell(s) \rangle ds \\
& \, + \E\int_0^T e^{-\lambda s}\big[ 2 \big(C(u_\ell(s),v_\ell(s)), C(z(s),w(s)) \big)_{l^2(H)} \\
& \, - |C(z(s),w(s))|_{l^2(H)}^2 
- 2\lambda (v_\ell(s), w(s)) + \lambda|w(s)|^2 \\
& \, - 2\lambda (u_\ell(s),z(s))_B 
+ \lambda|z(s)|_B^2 \big]ds \\
& \, - \E\int_0^T 2e^{-\lambda s} \big[ \langle Aw(s), v_\ell(s) - w(s)\rangle +  \langle Av_\ell(s),w(s)\rangle \big] ds + \bar{R}_\ell.
\end{split}
\end{equation}
We can now take the limit inferior along the subsequence $\ell'$.
Due to  Lemma~\ref{lemma:weaklimits}
and due to the weak sequential lower-semicontinuity of the norm, we see that
\begin{equation}
\label{eq:mono_trick_1}
\begin{split}
& e^{-\lambda T} \E\big(|v(T)|^2 + |u(T)|_B^2\big)
\leq \liminf_{\ell' \to \infty} e^{-\lambda T}
\E\big(|v_{\ell'}(T)|^2 + |u_{\ell'}(T)|_B^2\big)
\\
& \leq \E\big(|v_0|^2 + |u_0|_B^2\big) + 2\E\int_0^T \!\!e^{-\lambda s} \langle f(s),v(s) \rangle ds \\
& \quad + \E\int_0^T e^{-\lambda s}\big[ 2 \big(\bar{c}(s), C(z(s),w(s)) \big)_{l^2(H)} - |C(z(s),w(s))|_{l^2(H)}^2 \\
&    \quad - 2\lambda (v(s), w(s)) + \lambda|w(s)|^2 - 2\lambda (u(s),z(s))_B + \lambda|z(s)|_B^2 \big]ds \\
& \quad - \E\int_0^T 2e^{-\lambda s} \big[ \langle Aw(s), v(s) - w(s)\rangle +  \langle a(s),w(s)\rangle \big] ds.
\end{split}
\end{equation}
We now need the limit equation obtained in Lemma~\ref{lemma:eqns}
to proceed.
Taking expectation in~\eqref{eq:7} and
using Lemma~\ref{lemma:int_by_pts_cals}, we get
\begin{equation}
\label{eq:mono_trick_2}
\begin{split}
 e^{-\lambda T}\E \big(|v(T)|^2 + & |u(T)|_B^2 \big)  =  \E\big(|v_0|^2 + |u_0|_B^2\big)\\
& - \lambda \E \int_0^T e^{-\lambda s} \big[|v(s)|^2 + |u(s)|_B^2\big] ds\\
& + \E\int_0^T e^{-\lambda s} \big[ 2\langle f(s) - a(s),v(s)\rangle +
|\bar{c}(s)|_{l^2(H)}^2\big] ds.
\end{split}
\end{equation}
Subtracting~\eqref{eq:mono_trick_2} from~\eqref{eq:mono_trick_1} leads to
\begin{equation}
\label{eq:mono_trick_3}
\begin{split}
0
\leq{} & \liminf_{\ell' \to \infty} e^{-\lambda T}
\E\big(|v_{\ell'}(T)|^2 + |u_{\ell'}(T)|_B^2\big)
- e^{-\lambda T} \E\big(|v(T)|^2 + |u(T)|_B^2\big) \\
\leq{} & \, \E \int_0^T e^{-\lambda s} \big[ -|\bar{c}(s) - C(z(s),w(s))|_{l^2(H)}^2 \\
& + \lambda |v(s)-w(s)|^2 + \lambda |u(s)-z(s)|_B^2 + 2\langle a(s), v(s) - w(s)\rangle \big] ds\\
&- 2\E\int_0^T e^{-\lambda s} \langle Aw(s), v(s)-w(s)\rangle ds .
\end{split}
\end{equation}
This implies
\begin{equation}
\label{eq:mono_trick_4}
\begin{split}
& 2\E\int_0^T  e^{-\lambda s} \langle Aw(s), v(s)-w(s)\rangle ds  \\
 & \le  \E \int_0^T e^{-\lambda s} \big[ -|\bar{c}(s) - C(z(s),w(s))|_{l^2(H)}^2 \\
&\quad + \lambda |v(s)-w(s)|^2 + \lambda |u(s)-z(s)|_B^2 + 2\langle a(s), v(s) - w(s)\rangle \big] ds\\
& \leq  \E \int_0^T e^{-\lambda s} \big[
\lambda |v(s)-w(s)|^2 + \lambda |u(s)-z(s)|_B^2 + 2\langle a(s), v(s) - w(s)\rangle \big] ds
\end{split}
\end{equation}
Now we are ready to identify the limits.
First we take $w = v$ and $z = u$.
The first inequality in~\eqref{eq:mono_trick_4} leads to
\begin{equation*}
0 \leq -\E \int_0^T e^{-\lambda s} |\bar{c}(s) - C(u(s),v(s))|_{l^2(H)}^2 ds
\end{equation*}
which can only be true if $\bar{c}=C(u,v)$.
Next we take an arbitrary $\bar{w} \in \mathcal{L}^p(V)$, set $\bar{z} = u_0 + K\bar{w}$ and let $\epsilon \in (0,1)$.
Then with $w = v - \epsilon \bar{w}$ and $z = u - \epsilon \bar{z}$,
the second inequality in~\eqref{eq:mono_trick_4} leads to
\begin{equation*}
\begin{split}
&2\E\int_0^T e^{-\lambda s} \langle A(v(s)-\epsilon \bar{w}(s)), \epsilon \bar{w}(s)\rangle ds \\
& \leq  \E \int_0^T e^{-\lambda s} \big[ \lambda \epsilon^2 (|\bar{w}(s)|^2 + |\bar{z}(s)|_B^2) + 2\langle a(s), \epsilon \bar{w}(s)\rangle \big] ds.\\
\end{split}
\end{equation*}
We divide by $\epsilon > 0$.
Due to the hemicontinuity and growth assumptions on $A$
and since $\epsilon < 1$, we can apply
Lebesgue's theorem on dominated convergence and let $\epsilon \to 0$.
Hence, we arrive at
\begin{equation*}
\E\int_0^T e^{-\lambda s} \langle Av(s), \bar{w}(s)\rangle ds \leq \E \int_0^T e^{-\lambda s} \langle a(s), \bar{w}(s)\rangle  ds ,
\end{equation*}
which can only hold true for all $\bar{w} \in \mathcal{L}^p(V)$ if $a=Av$.
Finally,  we note that the uniqueness of the solution to equation~\eqref{eq:1} implies that the whole sequence converges to the limit and not only the subsequence.

We will now show that $v_\ell(T) \to v(T)$ in $L^2(\Omega;H)$
and $u_\ell(T) \to u(T)$ in $L^2(\Omega;V_B)$ as $\ell\to\infty$.
We first take the limit superior in~\eqref{eq:mono_trick_0} 
with $w=v$ and $z=u$ to obtain
\begin{equation*}
\begin{split}
\limsup_{\ell\to \infty} e^{-\lambda T} & \E \big( |v_\ell(T)|^2 + |u_\ell(T)|_B^2 \big) \\
& \leq \E\big(|v_0|^2 + |u_0|_B^2\big) + \E\int_0^T \!\!e^{-\lambda s} \Big[ 2\langle f(s)-a(s),v(s) \rangle  \\
& \quad + 2 \big(\bar{c}(s), C(u(s),v(s)) \big)_{l^2(H)} - |C(u(s),v(s))|_{l^2(H)}^2 \\
&    \quad - 2\lambda (v(s), v(s)) + \lambda|v(s)|^2 - 2\lambda (u(s),u(s))_B + \lambda|u(s)|_B^2 \Big]\,ds. \\
\end{split}
\end{equation*}
Since $a=Av$ and $\bar{c}=C(u,v)$  and due to~\eqref{eq:mono_trick_2}, we get
\begin{equation}
\label{eq:endptconv}
\begin{split}
\limsup_{\ell\to \infty} e^{-\lambda T} & \E \big( |v_\ell(T)|^2 + |u_\ell(T)|_B^2 \big)\\
& \leq \E\big(|v_0|^2 + |u_0|_B^2\big) + \E\int_0^T \!\! e^{-\lambda s} 
\Big[2\langle f(s)-Av(s),v(s) \rangle  \\
& \quad +|C(u(s),v(s))|_{l^2(H)}^2 
- \lambda|v(s)|^2 - \lambda|u(s)|_B^2  \Big]\, ds \\
& = e^{-\lambda T}\E \big( |v(T)|^2 + |u(T)|_B^2 \big). 
\end{split}
\end{equation}
Finally, due to weak sequential lower-semicontinuity of the norm 
and with~\eqref{eq:endptconv}, we see that
\[
\begin{split}
e^{-\lambda T}\E \big( |v(T)|^2 + |u(T)|_B^2 \big) 
& \leq \liminf_{\ell\to \infty} e^{-\lambda T}  \E \big( |v_\ell(T)|^2 + |u_\ell(T)|_B^2 \big) \\	
& \leq \limsup_{\ell\to \infty} e^{-\lambda T}  \E \big( |v_\ell(T)|^2 + |u_\ell(T)|_B^2 \big)\\
& \leq e^{-\lambda T}\E \big( |v(T)|^2 + |u(T)|_B^2 \big).
\end{split}
\]
Hence $\E \big( |v_\ell(T)|^2 + |u_\ell(T)|_B^2 \big) \to \E\big(|v(T)|^2 + |u(T)|_B^2\big)$ as $\ell \to \infty$.
The space $L^2(\Omega; (H,V_B))$ with the natural inner product is a Hilbert
space.
This is because the space $V_B$, under the conditions imposed on $B$, is 
a Hilbert space.
We can now use this together with the weak convergence
$v_\ell(T) \rightharpoonup v(T)$ in $L^2(\Omega;H)$
and $u_\ell(T) \rightharpoonup u(T)$ in $L^2(\Omega;V_B)$
to complete the proof.
\end{proof}

\begin{remark}
\label{remark:stronger_monotonicity}
It is possible to show that if $A$ and $C$ jointly satisfy
some appropriate stronger monotonicity assumption
then $v_\ell \to v$ in $L^p((0,T)\times \Omega;V_A)$ as $\ell \to \infty$.
For example, if there is $\mu > 0$ such that, almost surely, for any $w,z \in V_A$ and $u,v \in V_B$
\begin{equation}
\label{eq:stronger_monotonicity}
\begin{split}
&\langle Aw - Az, w-z \rangle  +  \lambda_A |w-z|^2 \\
& \geq \mu\|w-z\|_{V_A}^p +  \frac{1}{2}|C(u,w) - C(v,z)|_{l^2(H)}^2 - \lambda_B|u-v|_B^2
\end{split}
\end{equation}
then $v_\ell \to v$ in $L^p((0,T)\times \Omega;V_A)$ as $\ell \to \infty$.
\end{remark}

Indeed with~\eqref{eq:stronger_monotonicity}, we obtain, instead of~\eqref{eq:mono_trick_0}, the following (we have taken $w=v$ and $z=u$):
\begin{equation*}
\begin{split}
& \mu \E\int_0^T e^{-\lambda s} \|v_\ell(s) - v(s)\|_{V_A}^p ds
+ e^{-\lambda T} \E\big(|v_\ell(T)|^2 + |u_\ell(T)|_B^2\big)\\
& \leq \E\big(|v^0_\ell|^2 + |u^0_\ell|_B^2\big)
+ \E\int_0^T e^{-\lambda s}\big[ 2\langle f_\ell(s),v_\ell(s) \rangle ds \\
& \, + \E\int_0^T e^{-\lambda s}\big[ 2 \big(C^{r_\ell}(u_\ell(s),v_\ell(s)), C^{r_\ell}(u(s),v(s)) \big)_{l^2(H)}\\
& \, - |C^{r_\ell}(u(s),v(s))|_{l^2(H)}^2 
- 2\lambda (v_\ell(s), v(s)) + \lambda|v(s)|^2 - 2\lambda (u_\ell(s),u(s))_B\\
& \, + \lambda|u(s)|_B^2 \big]ds 
- \E\int_0^T 2e^{-\lambda s} \big[ \langle Av(s), v_\ell(s) - v(s)\rangle +  \langle Av_\ell(s),v(s)\rangle \big] ds + \bar{R}_\ell.
\end{split}
\end{equation*}
Taking the limit as $\ell \to \infty$ and using Lemma~\ref{lemma:weaklimits} together with the fact, established earlier, that $a=Av$ and $c=C(u,v)$, we obtain
\begin{equation*}
\begin{split}
& \mu \lim_{\ell \to \infty} \E\int_0^T  e^{-\lambda s} \|v_\ell(s) - v(s)\|_{V_A}^p ds + e^{-\lambda T}  \E\big(|v(T)|^2 + |u(T)|_B^2\big)\\
& \leq \E\big(|v_0|^2 + |u_0|_B^2\big) -\lambda \E\int_0^T e^{-\lambda s} \big[|v(s)|^2 + |u(s)|_B^2\big] ds\\
& \quad + \E\int_0^T e^{-\lambda s}\big[ 2\langle f(s) - Av(s),v(s) \rangle + |C(u(s),v(s))|_{l^2(H)}^2\big] ds.
\end{split}
\end{equation*}
If we subtract~\eqref{eq:mono_trick_2} then we obtain
\begin{equation*}
\mu \lim_{\ell \to \infty} \E\int_0^T  e^{-\lambda s} \|v_\ell(s) - v(s)\|_{V_A}^p ds \leq 0.
\end{equation*}
From this, we conclude that $v_\ell \to v$ in
$L^p((0,T)\times \Omega;V_A)$ and thus also $u_\ell \to u$ in $L^2((0,T)\times \Omega;V_B)$
as $\ell \to \infty$.

\section{Proof of uniqueness}
\label{sec:uniq}
In this short section, we will prove that the solution to~\eqref{eq:1}
is unique in the sense specified in Theorem~\ref{lemma:uniq}.
\begin{proof}[Proof of Theorem~\ref{lemma:uniq}]
Let $v := v_1 - v_2$ and $u := u_1 - u_2$.
Then $\P$-almost everywhere and for all $t\in [0,T]$
\begin{equation*}
\begin{split}
v(t) ={}&   - \int_0^t \big[Av_1(s) - Av_2(s) + Bu(s) \big] ds \\
& + \int_0^t \left[C(u_1(s),v_1(s)) - C(u_2(s),v_2(s))\right] dW(s)  	
\end{split}
\end{equation*}
holds in $V_A^*$.
With the assumption $V_A \hookrightarrow V_B$, we may apply It\^o's formula for the square of the norm (see, e.g., Pr\'ev\^ot and R\"ockner~\cite[Theorem 4.2.5]{prevot:rockner:concise}) and obtain
\begin{equation*}
\begin{split}
|v(t)|^2 ={} & -2\int_0^t  \langle Av_1(s) - Av_2(s)+Bu(s),v(s) \rangle ds \\
& + 2\int_0^t (v(s), C(u_1(s),v_1(s)) - C(u_2(s), v_2(s)) dW(s) )\\
& + \int_0^t |C(u_1(s),v_1(s)) - C(u_2(s), v_2(s))|_{l^2(H)}^2 ds.
\end{split}
\end{equation*}
Since $u(0) = 0$, we obtain with (\ref{eq:Buv})
\begin{equation*}
\begin{split}
|v(t)|^2 + |u(t)|_B^2 ={}& -2\int_0^t  \langle Av_1(s) - Av_2(s),v(s) \rangle ds \\
& + 2\int_0^t \big(v(s), [C(u_1(s),v_1(s)) - C(u_2(s), v_2(s))] dW(s) \big) \\
& + \int_0^t |C(u_1(s),v_1(s)) - C(u_2(s), v_2(s))|_{l^2(H)}^2 ds.
\end{split}
\end{equation*}
Now we apply It\^o's formula for real-valued processes (similar to Lemma~\ref{lemma:int_by_pts_cals}) to obtain
\begin{equation*}
\begin{split}
e^{-\lambda t}\big(|v(t)|^2 + & |u(t)|_B^2\big) = - \lambda \int_0^t e^{-\lambda s} \big(|v(s)|^2 + |u(s)|_B^2 \big)ds\\
& - 2\int_0^t e^{-\lambda s} \langle Av_1(s) - Av_2(s),v(s) \rangle ds \\
& + \int_0^t e^{-\lambda s} |C(u_1(s),v_1(s)) - C(u_2(s), v_2(s))|_{l^2(H)}^2 ds + m(t),
\end{split}
\end{equation*}
where
\begin{equation*}
m(t) = 2\int_0^t e^{-\lambda s} \big(v(s), [C(u_1(s),v_1(s)) - C(u_2(s), v_2(s))] dW(s)\big).
\end{equation*}
This together with (\ref{eq:mod_monotonicity}) yields
\begin{equation*}
  0 \leq e^{-\lambda t}\big(|v(t)|^2  + |u(t)|_B^2\big) \leq m(t).
\end{equation*}
Hence the process $m(t)$ is non-negative for all $t\in [0,T]$.
We also can see that it is a continuous
local martingale starting from $0$.
Thus, almost surely, $m(t) = 0$ for all $t\in [0,T]$.
But this in turn means that, almost surely, $|v_1(t) - v_2(t)|^2 = |v(t)|^2 = 0$ as well as $|u_1(t) - u_2(t)|_B^2 = |u(t)|_B^2 = 0$ for all $t\in [0,T]$.
Thus solutions  to~\eqref{eq:1} must be indistinguishable.
\end{proof}

\section*{Acknowledgements}
The authors would like to thank Raphael Kruse (Berlin) for helpful  discussions and comments 
and to the referees for their careful reading and helpful suggestions.

\end{document}